\documentclass[10pt,letterpaper,reqno]{amsart}
\usepackage[T1]{fontenc}
\usepackage{amsmath,amssymb,amsthm,mathtools}
\usepackage{xcolor}
\usepackage{hyperref}
\usepackage{doi}
\hypersetup{colorlinks=true,linkcolor=blue,citecolor=blue,urlcolor=blue}
\usepackage{indentfirst}
\usepackage{mathrsfs}
\usepackage{setspace}
\usepackage{color}
\usepackage{tikz}
\usepackage{float}
\usepackage{cite}
\usepackage{bm}

\usepackage{lineno}
\usepackage{latexsym}
\usepackage{verbatim}
\usepackage{float}
\usepackage{cases}
\usepackage{esint}

\theoremstyle{plain}
\newtheorem{thm}{Theorem}[section]
\newtheorem{cor}[thm]{Corollary}
\newtheorem{lem}[thm]{Lemma}

\theoremstyle{definition}
\newtheorem{defin}[thm]{Definition}

\theoremstyle{remark}
\newtheorem{rem}[thm]{Remark}

\newcommand{\tr}{\operatorname{Tr}}
\newcommand{\diag}{\text{diag}}

\allowdisplaybreaks[4] 
  \numberwithin{equation}{section}
  \numberwithin{figure}{section}

\begin{document}

\title[Sharp perturbation bounds]
{Sharp Frobenius-norm perturbation bounds for subunitary and positive
	polar factors}

	\author[Z.~Xu]{Zongben Xu}

\address{School of Mathematics and Statistics, Xi'an Jiaotong University,
	Xi'an 710049, P.~R.~China}
\email{zbxu@mail.xjtu.edu.cn}

\author{Teng Zhang}

\address{School of Mathematics and Statistics, Xi'an Jiaotong
	University, Xi'an 710049, P.~R.~China}
\email{teng.zhang@stu.xjtu.edu.cn}

\begin{abstract}
Leveraging tools from convex analysis and incorporating additional singular value information of matrices, we derive sharp upper and lower Frobenius-norm perturbation bounds
for the subunitary factor in the canonical generalized polar
decomposition. We also obtain sharp singular-value-dependent,
dimension-free bounds for the positive factor and characterize their
attainment and limiting cases. These results refine the bounds of Li
and Sun [\emph{SIAM J. Matrix Anal. Appl.}, 23 (2002),
pp.~1183--1193] and the well-known Araki--Yamagami inequality
[\emph{Comm. Math. Phys.}, 81 (1981), no.~1, pp.~89--96].
The same method yields a singular-value sharpening of Lee's conjecture on the 
Frobenius norm, together with a sharp lower counterpart.
We further obtain a singular-value refinement of the matrix
arithmetic--geometric mean inequality, give an independent proof of
the von Neumann trace inequality, and establish a sharp lower
counterpart to Kittaneh's inequality for normal matrices.
\end{abstract}

\subjclass[2020]{ 15A45, 15A60, 47A30, 47A50, 65F10}

\keywords{Perturbation bounds, Frobenius norm, generalized polar
	decomposition, subunitary polar factor, positive polar factor}

\date{}
\maketitle
\section{Introduction}
Let $\mathbb{C}^{m\times n}$ denote the set of $m\times n$ complex matrices, and let $\mathbb{C}_r^{m\times n}$ be the subset of those matrices of rank $r$. We denote the Frobenius norm on $\mathbb{C}^{m\times n}$ by $\|\cdot\|_F$, the trace of a square matrix $A $ by $\tr\, A$, and the conjugate transpose of $A$ by $A^*$. The absolute value of $A$ is defined as $|A| := (A^*A)^{1/2}$.

For any given matrix $A\in \mathbb{C}_r^{m\times n}$, there exist a subunitary matrix $Q\in \mathbb{C}_r^{m\times n}$ and  a positive semidefinite matrix $H\in \mathbb{C}_r^{n\times n} $ such that
\begin{eqnarray*}
	A=QH.
\end{eqnarray*}
This factorization is called a generalized polar decomposition of \(A\).
The matrix \(Q\) is a partial isometry (also called a subunitary
matrix), and \(H\) is positive semidefinite. In general, such a
factorization need not be unique. Throughout this paper, however, we
always use the canonical generalized polar decomposition, characterized by
$
\mathcal R(Q^*)=\mathcal R(H),
$
where \(\mathcal R(\cdot)\) denotes the column space. Under this condition
the decomposition is unique, \(H=|A|\), and, for a compact singular value
decomposition \(A=U_1\Sigma_1V_1^*\), one has
\[
Q=U_1V_1^*,
\qquad
H=V_1\Sigma_1V_1^*=|A|.
\]
See \cite{SC89}.

The generalized polar decomposition of a matrix plays a key role in numerous fields, including scientific computing, optimization theory, aerospace, and even psychometrics, as evidenced by \cite{GV96, Hig86,HMZ14,KL91}.  
For this reason, the perturbation theory of the generalized polar decomposition  of  matrices has garnered substantial attention from researchers, as documented in \cite{Bar89,Bha97,CL05,CL06,CL08,LiR93,LiR95,LiR97,LiR05,LS02,LS03,LiW05,LS06,LiW08,HMZ14, Mat93,SC89,Zhu18}.  Extensive investigations into the Frobenius norm have been conducted in the literature.  

Let \begin{equation}\label{gpc}
	A=QH \text{ and } \widetilde{A}=\widetilde{Q}\widetilde{H}
\end{equation}
be the generalized polar decompositions of $A\in \mathbb{C}_r^{m\times n}$ and $\widetilde{A}\in\mathbb{C}_s^{m\times n}$, respectively. In fact,  $H=\left|A\right|$ and $\widetilde{H}=|\widetilde{A}|$ here. Let the singular values of $A$ and $\widetilde{A}$, arranged in decreasing order, be  $\sigma_1\ge \ldots\ge \sigma_r>0$ and $\widetilde{\sigma}_1\ge \ldots\ge \widetilde{\sigma}_s>0$, respectively. When
\(\operatorname{rank}(A)=\operatorname{rank}(\widetilde A)\),
 the best known previous perturbation bound for the Frobenius norm  of the subunitary  polar factors is attributed to Li and Sun \cite{LS02}. They noted that their bound is sharp in certain cases and provided many relevant examples.
\begin{thm}[Li--Sun]\label{Li--Sun}
	Let $A, \widetilde{A}=A+E\in \mathbb{C}_r^{m\times n}$ have the generalized polar decompositions (\ref{gpc}). Then
	\begin{eqnarray}\label{lsbound}
		\|Q-\widetilde{Q}\|_F\le \dfrac{2}{\sigma_r+\widetilde{\sigma}_r}\|E\|_F.
	\end{eqnarray}
\end{thm}
Li \cite[Remark~3.2]{LiW08} conjectured that the Frobenius norm in Theorem~\ref{Li--Sun} can be replaced by an arbitrary
unitarily invariant norm. However, the assertion already fails for the spectral
norm $\|\cdot\|_\infty$. Let
$
A=Q=\operatorname{diag}(1,1,0),
$
and let
\[
\widetilde Q=
\begin{pmatrix}
	\frac{\sqrt3}{2} & \frac12 & 0\\
	-\frac{\sqrt3}{4} & \frac34 & 0\\
	\frac14 & -\frac{\sqrt3}{4} & 0
\end{pmatrix},
\quad
\widetilde A=
\begin{pmatrix}
	\frac{1+9\sqrt3}{16} & \frac{9-\sqrt3}{16} & 0\\
	-\frac{3+9\sqrt3}{32} & \frac{27-\sqrt3}{32} & 0\\
	\frac{9+\sqrt3}{32} & \frac{1-9\sqrt3}{32} & 0
\end{pmatrix}.
\] A direct calculation gives
\[
\widetilde Q^*\widetilde Q
=
\operatorname{diag}(1,1,0),
\qquad
\widetilde H:=\widetilde Q^*\widetilde A
=
\begin{pmatrix}
	\frac98+\frac{\sqrt3}{16}&-\frac1{16}&0\\
	-\frac1{16}&\frac98-\frac{\sqrt3}{16}&0\\
	0&0&0
\end{pmatrix}
\ge 0.
\]
Moreover,
\[
\widetilde A=\widetilde Q\widetilde H,
\qquad
\sigma(\widetilde H)=\left(\frac54,1,0\right),
\qquad
\mathcal R(\widetilde Q^*)=\mathcal R(\widetilde H).
\]
Hence \(\widetilde Q\) is the canonical subunitary polar factor of
\(\widetilde A\).
Then \(\sigma(A)=(1,1)\) and
\(\sigma(\widetilde A)=(5/4,1)\), so
\(2/(\sigma_2+\widetilde\sigma_2)=1\). However,
\[
\|Q-\widetilde Q\|_\infty^2
=
\frac{9-4\sqrt3}{4}
\approx 0.5179491924,
\]
whereas
\[
\|A-\widetilde A\|_\infty^2
=
\frac{11}{8}
-\frac{17\sqrt3}{32}
+
\frac{\sqrt{1999-1154\sqrt3}}{32}
\approx 0.4692829401.
\]
Thus
\[
\|Q-\widetilde Q\|_\infty>\|A-\widetilde A\|_\infty,
\]
which disproves the conjecture.

When \(\operatorname{rank}(A)\neq \operatorname{rank}(\widetilde A)\), the situation
is fundamentally different. Indeed, let \(Q\) and \(\widetilde Q\) be the
corresponding subunitary polar factors, and set
\[
r=\operatorname{rank}(A),
\qquad
s=\operatorname{rank}(\widetilde A).
\]
Since the singular values of \(Q\) and \(\widetilde Q\) consist of \(r\) and
\(s\) ones, respectively, followed by zeros, Mirsky's theorem \cite[p.~78, Problem~III.6.13; see also p.~82]{Bha97} implies that,
for every normalized unitarily invariant norm \(\|\cdot\|\),
\[
\| Q-\widetilde Q\|
\ge
\|
\operatorname{diag}
\bigl(
\underbrace{1,\ldots,1}_{|r-s|},
0,\ldots,0
\bigr)
\|
\ge 1.
\]
In particular,
\[
\|Q-\widetilde Q\|_{F}\ge \sqrt{|r-s|}.
\]
Thus, a change of rank produces an unavoidable jump in the subunitary polar
factor, and no rank-independent continuity estimate can hold across matrices
of different ranks. Perturbation bounds in this regime should therefore be
understood as quantifying the precise magnitude of this intrinsic
instability. Chen and Li \cite{CL08} established the following estimate and
further demonstrated, by means of an example, that their bound is optimal in
a certain sense.
\begin{thm}[Chen--Li] \label{Chen-Li}
		Let $r\le s$ and $A\in \mathbb{C}_r^{m\times n}, \widetilde{A}=A+E\in \mathbb{C}_s^{m\times n}$ have the generalized polar decompositions (\ref{gpc}).  Then
	\begin{eqnarray*}
		\|Q-\widetilde{Q}\|_F\le \frac{1}{\min\{\sigma_r,\widetilde{\sigma}_s\}}\|E\|_F.
	\end{eqnarray*}
\end{thm}

Regarding positive polar factors, the well-known Araki--Yamagami inequality \cite{AY81} states that  
\begin{thm}[Araki--Yamagami] \label{thm1.2}
	Let $A, \widetilde{A}=A+E\in \mathbb{C}^{m\times n}$ have the generalized polar decompositions (\ref{gpc}). Then
	\begin{eqnarray}\label{akbound}
		\|H-\widetilde{H}\|_F\le \sqrt{2}\|E\|_F,
	\end{eqnarray}
	where $\sqrt{2}$ is the optimal constant.
\end{thm}

Drawing on Lin--Zhang's proof of Lee's conjecture \cite{LZ22}, we
give a new proof of Theorem~\ref{thm1.2} in Section~\ref{s3}. In the
full-column-rank case, Chen and Li \cite{CL08} obtained a refinement
of the Araki--Yamagami bound \eqref{akbound}. Further additive
perturbation bounds for generalized polar decompositions, subunitary
polar factors, and positive polar factors were obtained in
\cite{CLS04,ZYL13,LYS13}. To the best of our knowledge, however, no
sharp dimension-free Frobenius-norm refinement depending only on the
complete singular-value lists and valid for arbitrary
\(1\le r\le s\) was previously available.

It should be noted that perturbation bounds for (sub)unitary and positive polar factors depend heavily on the number field, rank, and dimension of  matrices. As shown in Li--Sun's bound (\ref{lsbound}) and Araki--Yamagami's bound (\ref{akbound}), the introduced singular value information is far from sufficient, leaving considerable room for improvement. However, for such a basic and important problem in generalized polar decompositions of matrices, little progress has been made over the past decade.

Our sharp upper perturbation bounds for the Frobenius norm of subunitary and positive polar factors are given as follows.
\begin{thm}\label{thm1.3}
	Let $r\le s$ and $A\in \mathbb{C}_r^{m\times n}, \widetilde{A}=A+E\in \mathbb{C}_s^{m\times n}$ have the generalized polar decompositions (\ref{gpc}). Then
	\begin{eqnarray*}
		\|Q-\widetilde{Q}\|_F\le \sqrt{\max\limits_{0\le k\le r}\frac{s-r + 4k}
			{\sum_{j=1}^{r-k}(\sigma_j-\widetilde{\sigma}_j)^2+\sum_{j=1}^{k}(\sigma_{r+1-j}+\widetilde{\sigma}_{s-k+j})^2+\sum_{j=r-k+1}^{s-k}\widetilde{\sigma}_j^2}}\|E\|_F,
	\end{eqnarray*}
Any candidate fraction with zero denominator is omitted.
The coefficient is sharp for every admissible pair of singular-value
lists. More precisely, if \(k^\star\) is any index at which the
maximum is attained, then the unitary construction given in the proof
yields equality with \(\|E\|_F>0\).
\end{thm}

\begin{thm}\label{thm1.4}
	Let $r\le s$ and $A\in \mathbb{C}_r^{m\times n}, \widetilde{A}=A+E\in \mathbb{C}_s^{m\times n}$ have the generalized polar decompositions (\ref{gpc}). Then
	\begin{eqnarray*}
		\|H-\widetilde{H}\|_F\le \sqrt{\frac{F_{r,s}-\sqrt{F_{r,s}^2-2G_{r,r}F_{r,s}}}{G_{r,r}}}\|E\|_F,
	\end{eqnarray*}
where
$
F_{r,s}
:=
\sum_{j=1}^r\sigma_j^2+
\sum_{j=1}^s\widetilde{\sigma}_j^2,$
and $
G_{r,r}
:=
\sum_{j=1}^r\sigma_j\widetilde{\sigma}_j.
$
The coefficient is sharp as a dimension-free bound depending only on
the displayed singular values. More precisely, if \(n-s\ge r\) and
\(F_{r,s}>2G_{r,r}\), equality is attained. If \(n-s\ge r\) and
\(F_{r,s}=2G_{r,r}\), equality cannot occur for a nonzero perturbation
\(E\); nevertheless, the coefficient is sharp in the limiting sense.
\end{thm}
\begin{rem}
No single index \(k\) is optimal for all admissible singular-value lists in Theorem \ref{thm1.3}; see Table \ref{tab:max_k_examples} in Section \ref{s2} for illustrative examples.
\end{rem}
For $r=s$ in Theorem \ref{thm1.3}, we provide a refinement of Li--Sun's bound (\ref{lsbound}). It suffices to see that  $\sum_{j=1}^{k}(\widetilde{\sigma}_{r-k+j}+\sigma_{r+1-j})^2\ge k(\sigma_r+\widetilde{\sigma}_r)^2$ for any $1\le k\le r$.
\begin{cor}\label{cor1}
	Let  $A, \widetilde{A}=A+E\in \mathbb{C}_r^{m\times n}$ have the generalized polar decompositions (\ref{gpc}). Then
	\begin{eqnarray}\label{nb1}
		\|Q-\widetilde{Q}\|_F\le \sqrt{\max\limits_{1\le k\le r}\frac{4k}
			{\sum_{j=1}^{r-k}(\sigma_j-\widetilde{\sigma}_j)^2+\sum_{j=1}^{k}(\sigma_{r+1-j}+\widetilde{\sigma}_{r-k+j})^2}}\|E\|_F,
	\end{eqnarray}
where the coefficient is sharp. More precisely, for every admissible
pair of singular-value lists, if \(k^\star\) is an index at which the
maximum is attained, then the equality construction in the proof of
Theorem~\ref{thm1.3}, specialized to \(r=s\), attains equality in
\eqref{nb1}.
\end{cor}

Let $k^*$ be the optimal index of the right-hand side of (\ref{nb1}). It is easy to see that when $\sigma_j=\widetilde{\sigma}_j, 1\le j\le r- k^*; \sigma_{r-k^*+1}=\ldots=\sigma_r$ and $\widetilde{\sigma}_{r-k^*+1}=\ldots=\widetilde{\sigma}_r$, the bound (\ref{nb1}) reduces to Li--Sun's bound (\ref{lsbound}).
\begin{rem}
	Theorem \ref{thm1.3} is an improvement of Theorem \ref{Chen-Li}.
\end{rem}
\begin{proof}
It is sufficient to note that
	\begin{eqnarray*}
	&&	\sum_{j=1}^{r-k}(\sigma_j-\widetilde{\sigma}_j)^2+\sum_{j=1}^{k}(\sigma_{r+1-j}+\widetilde{\sigma}_{s-k+j})^2+\sum_{j=r-k+1}^{s-k}\widetilde{\sigma}_j^2\\
	&\ge&k(\sigma_{r}+\widetilde{\sigma}_{s})^2+(s-r)\widetilde{\sigma}_s^2\\
	&\ge&(4k+s-r)\min\{\sigma_r^2, \widetilde{\sigma}_s^2\}.
	\end{eqnarray*}
\end{proof}
\begin{rem}
 Theorem \ref{thm1.4} is a refinement of Theorem \ref{thm1.2}.
\end{rem}
\begin{proof}
	Set $t={F_{r,s}}/{G_{r,r}}$. Clearly, by arithmetic-geometric mean (AM-GM) inequality, $t\ge 2$, with strict inequality if $r<s$ or $r=s$ but $\sigma_j\neq \widetilde{\sigma}_j$ for some $j$. The coefficient squared in  Theorem \ref{thm1.4}
	\begin{eqnarray*}
		\frac{F_{r,s}-\sqrt{F_{r,s}^2-2G_{r,r}F_{r,s}}}{G_{r,r}}=t-\sqrt{t^2-2t}\overset{\triangle}{=} g_1(t),
	\end{eqnarray*}
For \(t>2\),
\[
g_1'(t)
=
\frac{\sqrt{t^2-2t}-(t-1)}
{\sqrt{t^2-2t}}
<0.
\]
Since \(g_1\) is continuous at \(t=2\), it follows that
\[
g_1(t)\le g_1(2)=2,
\qquad t\ge2.
\]
\end{proof}
\begin{rem}
		Let  $A\in \mathbb{C}_r^{m\times n},\widetilde{A}\in \mathbb{C}_s^{m\times n}$ have the generalized polar decompositions (\ref{gpc}) and $A, \widetilde{A}\neq 0$.  Then for $r\neq s$ or $r=s$ but $\sigma_j\neq \widetilde{\sigma}_j$ for some $j$, we have
		\[
		\|H-\widetilde{H}\|_F< \sqrt{2}\|E\|_F.
		\]
\end{rem}
We also derive sharp perturbation lower bounds for the Frobenius norm of subunitary polar factors and positive polar factors.
\begin{thm}\label{thm1.7}
		Let $r\le s$ and $A\in \mathbb{C}_r^{m\times n}, \widetilde{A}=A+E\in \mathbb{C}_s^{m\times n}$ have the generalized polar decompositions (\ref{gpc}). Then 
	\begin{eqnarray*}
		\|Q-\widetilde{Q}\|_F\ge \sqrt{\min\limits_{0\le k\le r}\frac{s-r + 4k}
			{\sum_{j=1}^{k}(\sigma_j+\widetilde{\sigma}_j)^2+\sum_{j=1}^{r-k}(\sigma_{r+1-j}-\widetilde{\sigma}_{s-r+k+j})^2+\sum_{j=k+1}^{s-r+k}\widetilde{\sigma}_j^2}}\|E\|_F,
	\end{eqnarray*}
where the coefficient is sharp for every admissible pair of
singular-value lists. More precisely, if \(k_\star\) is any index at
which the minimum is attained, then the unitary construction given in
the proof yields equality with \(\|E\|_F>0\).
\end{thm}
\begin{thm}\label{thm1.8}
		Let $r\le s$ and $A\in \mathbb{C}_r^{m\times n}, \widetilde{A}=A+E\in \mathbb{C}_s^{m\times n}$ have the generalized polar decompositions (\ref{gpc}). Then
	\begin{eqnarray*}
		\|H-\widetilde{H}\|_F\ge \sqrt{\dfrac{F_{r,s}-2G_{r,r}}{F_{r,s}+2G_{r,r}}}\|E\|_F,
	\end{eqnarray*}
where $F_{r,s}:=\sum_{j=1}^r\sigma_j^2+\sum_{j=1}^s\widetilde\sigma_j^2,
G_{r,r}:=\sum_{j=1}^r\sigma_j\widetilde\sigma_j$. The coefficient is sharp for every admissible pair of singular-value
lists.
\end{thm}

Regarding the relation between the Frobenius norms of $H+\widetilde{H}$ and
$A+\widetilde{A}$, Lee \cite{Lee10} posed a conjecture in 2010, which was
proved by Lin and Zhang \cite{LZ22} in 2022. Subsequently, the second author
\cite{Zha25} gave a new proof based on the matrix Cauchy--Schwarz inequality.
The conjecture was further extended by Tang and Zhang \cite{TZ26} to sums of
multiple matrices. Related developments concerning operator symmetric moduli
and their triangle inequalities can be found in \cite{Zha26a,Zha26b}.
\begin{thm}[Lee's conjecture]\label{lee}
	Let $A,\widetilde{A}\in \mathbb{C}^{m\times n}$ have the generalized polar decompositions (\ref{gpc}).  Then
	\begin{eqnarray*}
		\|A+\widetilde{A}\|_F\le \sqrt{\dfrac{1+\sqrt{2}}{2}}\|H+\widetilde{H}\|_F,
	\end{eqnarray*}
	where  $\sqrt{\frac{1+\sqrt{2}}{2}}$ is the optimal constant.
\end{thm}
We now establish the following sharp singular-value refinement of
Lee's conjecture.
\begin{thm}[A strong version of Lee's conjecture]\label{slee}
	Let $r\le s$ and $A\in \mathbb{C}_r^{m\times n}, \widetilde{A}=A+E\in \mathbb{C}_s^{m\times n}$ have the generalized polar decompositions (\ref{gpc}). Then
	\begin{eqnarray*}
		\|A+\widetilde{A}\|_F\le \sqrt{\dfrac{G_{r,r}}{\sqrt{F_{r,s}^2+2G_{r,r}F_{r,s}}-F_{r,s}}}\|H+\widetilde{H}\|_F,
	\end{eqnarray*}
where $F_{r,s}:=\sum_{j=1}^r\sigma_j^2+\sum_{j=1}^s\widetilde\sigma_j^2,
G_{r,r}:=\sum_{j=1}^r\sigma_j\widetilde\sigma_j$. The coefficient is sharp as a dimension-free bound depending only on the
listed singular values. In particular, the equality construction in the proof
requires enough ambient dimension, for instance \(n-s\ge r\).
\end{thm}
The corresponding sharp lower bound is given by the following.
\begin{thm}\label{cslee}
	Let $r\le s$ and $A\in \mathbb{C}_r^{m\times n}, \widetilde{A}=A+E\in \mathbb{C}_s^{m\times n}$ have the generalized polar decompositions (\ref{gpc}). Then
	\begin{eqnarray*}
		\|A+\widetilde{A}\|_F\ge \sqrt{\dfrac{F_{r,s}-2G_{r,r}}{F_{r,s}+2G_{r,r}}}\|H+\widetilde{H}\|_F,
	\end{eqnarray*}
	where $F_{r,s}:=\sum_{j=1}^r\sigma_j^2+\sum_{j=1}^s\widetilde\sigma_j^2,
	G_{r,r}:=\sum_{j=1}^r\sigma_j\widetilde\sigma_j$ and the coefficient  is optimal.
\end{thm}

\begin{rem}
	Theorem \ref{slee} is stronger than Theorem \ref{lee}.
\end{rem}
\begin{proof}
	Set $t={F_{r,s}}/{G_{r,r}}$. The  coefficient squared in Theorem \ref{slee}
	\begin{eqnarray*}
		\dfrac{G_{r,r}}{\sqrt{F_{r,s}^2+2G_{r,r}F_{r,s}}-F_{r,s}}&=&\dfrac{1}{\sqrt{t^2+2t}-t}\overset{\triangle}{=} g_2(t),
	\end{eqnarray*}
	since $g_2'(t)=\dfrac{\sqrt{t^2+2t}-(t+1)}{\sqrt{t^2+2t}\left(\sqrt{t^2+2t}-t \right)^2 }<0$, $g_2(t)\le g_2(2)=\frac{1}{2\sqrt{2}-2}=\frac{1+\sqrt{2}}{2}$.
\end{proof}
\begin{rem}
	Let  $A\in \mathbb{C}_r^{m\times n},\widetilde{A}\in \mathbb{C}_s^{m\times n}$ have the generalized polar decompositions (\ref{gpc})  and $A, \widetilde{A}\neq 0$.  Then for $r\neq s$ or $r=s$ but $\sigma_j\neq \widetilde{\sigma}_j$ for some $j$, we have
	\begin{eqnarray*}
		\|A+\widetilde{A}\|_F< \sqrt{\dfrac{1+\sqrt{2}}{2}}\|H+\widetilde{H}\|_F.
	\end{eqnarray*}
\end{rem}
Finally, we present several examples to demonstrate the potential of our method to address inequalities involving the Frobenius norm of matrices, such as refinements of the matrix arithmetic-geometric mean (AM-GM) inequality \cite[p. 263]{Bha97}, the classical von Neumann trace inequality viewed as a singular-value refinement of the Frobenius Cauchy--Schwarz inequality \cite{Bha97}, and giving a lower bound of Kittaneh's result \cite{Kit86}. For clarity, we assume the singular values of $B\in \mathbb{C}^{m\times n}_s$ are arranged in decreasing order as $\widehat{\sigma}_1\ge \ldots\ge \widehat{\sigma}_s>0$.

The classical AM-GM inequality involving the Frobenius norm of matrices states that 
\begin{thm}[AM-GM]
	Let $A,B\in \mathbb{C}^{m\times n}$. Then
	\[
	\|AB^*\|_F\le \frac{1}{2}\|\left|A\right|^2+\left|B\right|^2\|_F.
	\]
\end{thm}
We prove the following refinement.
\begin{thm}[Stronger AM-GM]\label{sag}
	Let $r\le s$ and $A\in\mathbb{C}_r^{m\times n},B\in \mathbb{C}_s^{m\times n}$. Then
	\[
	\|AB^*\|_F\le \left(\frac{\sum_{j=1}^r\sigma_j^2\widehat{\sigma}_j^2}{\sum_{j=1}^r\sigma_j^4+\sum_{j=1}^s\widehat{\sigma}_j^4+2\sum_{j=1}^r\sigma_j^2\widehat{\sigma}_j^2} \right)^\frac{1}{2}\|\left|A\right|^2+\left|B\right|^2\|_F,
	\]
	where the coefficient is optimal.
\end{thm}

The well-known Cauchy--Schwarz inequality involving the Frobenius norm of matrices states that 
\begin{thm}[Cauchy--Schwarz]
	Let $A,B\in \mathbb{C}^{m\times n}$. Then
	\[
	\left|\tr\, B^*A\right|\le \|A\|_F\|B\|_F.
	\]
\end{thm}
The following classical von Neumann trace inequality may be viewed as a
singular-value refinement of the Frobenius Cauchy--Schwarz inequality;
see, for example, \cite{Bha97}.
\begin{thm}[von Neumann trace inequality]\label{scs}
	Let $r\le s$ and  $A\in \mathbb{C}_r^{m\times n},B\in \mathbb{C}_s^{m\times n}$. Then
	\[
	\left|\tr\, B^*A\right|\le\frac{\sum_{j=1}^r\sigma_j\widehat{\sigma}_j}{\left(\sum_{j=1}^r\sigma_{j}^2 \right)^\frac{1}{2}\left(\sum_{j=1}^s\widehat{\sigma}_j^2 \right)^\frac{1}{2}} \|A\|_F\|B\|_F,
	\]
where the coefficient is optimal.
\end{thm}

Kittaneh \cite{Kit86} obtained the following refinement of the
Araki--Yamagami inequality (\ref{akbound}).

\begin{thm}[Kittaneh]\label{thmkit}
	Let $A,B\in \mathbb{C}^{m\times n}$. Then
	\[
	\|\left|A\right|-\left|B\right|\|_F^2+\|\left|A^*\right|-\left|B^*\right|\|_F^2\le 2\|A-B\|_F^2.
	\]
\end{thm}
In particular, when $A, B$ are normal, we have
\begin{cor}[Kittaneh]\label{kitcor}
Let $A,B\in \mathbb{C}^{n\times n}$ be two normal matrices. Then
\[
\|\left|A\right|-\left|B\right|\|_F\le \|A-B\|_F.
\]
\end{cor}
We remark that replacing $A, B$ by $\mathcal{A}=\begin{pmatrix}
	0&A\\
	A^*&0
\end{pmatrix}, \mathcal{B}=\begin{pmatrix}
	0&B\\
	B^*&0
\end{pmatrix}$ in Corollary \ref{kitcor} yields Theorem \ref{thmkit}. A new proof of Corollary \ref{kitcor} will be given in Section \ref{s3}. 

Let
\(\lambda_1,\ldots,\lambda_r\) and
\(\widehat{\lambda}_1,\ldots,\widehat{\lambda}_s\)
be the lists of nonzero eigenvalues of \(A\) and \(B\), respectively,
repeated according to algebraic multiplicity. For \(0\le k\le d\),
define
\[
\mathfrak I_k(d)
:=
\left\{
(i_1,\ldots,i_k)\in\{1,\ldots,d\}^k:
i_1,\ldots,i_k\ \text{are pairwise distinct}
\right\},
\]
with the convention that
$
\mathfrak I_0(d):=\{()\},
$
where \( () \) denotes the unique empty \(0\)-tuple.  The following theorem gives a sharp lower counterpart to
Corollary~\ref{kitcor}.
\begin{thm}\label{thmlast}
	Let \(1\le r\le s\le n\), and let
	\(A\in\mathbb C_r^{n\times n}\) and
	\(B\in\mathbb C_s^{n\times n}\) be normal matrices. Then
	\[
	\||A|-|B|\|_F
	\ge
	\min_{\substack{0\le k\le r\\
			(i_1,\ldots,i_k)\in\mathfrak I_k(s)\\
			(j_1,\ldots,j_k)\in\mathfrak I_k(r)}}
	\left(
	\frac{
		\widehat F_{r,s}
		-2\displaystyle\sum_{t=1}^k
		|\widehat\lambda_{i_t}||\lambda_{j_t}|
	}{
		\widehat F_{r,s}
		-2\displaystyle\sum_{t=1}^k
		\Re\!\left(
		\widehat\lambda_{i_t}\overline{\lambda_{j_t}}
		\right)
	}
	\right)^{1/2}
	\|A-B\|_F,
	\]
	where $
	\widehat F_{r,s}
	:=
	\sum_{j=1}^r|\lambda_j|^2
	+
	\sum_{j=1}^s|\widehat\lambda_j|^2. $
	Any candidate with zero denominator is omitted. Whenever the
	prescribed eigenvalue lists admit normal matrices \(A\) and \(B\)
	with \(A\ne B\), the coefficient is optimal. In the degenerate case
	in which every realization satisfies \(A=B\), the inequality is
	trivial and no nontrivial optimality assertion is made.
\end{thm}
 In particular, when one of \(A\) and \(B\) has full rank, Corollary \ref{kitcor} can be strengthened.
\begin{thm}\label{thmlast1}
	Let \(1\le r\le n\), and let
	\(A\in\mathbb C_r^{n\times n}\) and
	\(B\in\mathbb C_n^{n\times n}\) be normal matrices. Then
	\[
	\||A|-|B|\|_F
	\le
	\max_{(\pi_1,\ldots,\pi_r)\in\mathfrak I_r(n)}
	\left(
	\frac{
		\widehat F_{r,n}
		-2\displaystyle\sum_{j=1}^r
		|\widehat\lambda_{\pi_j}||\lambda_j|
	}{
		\widehat F_{r,n}
		-2\displaystyle\sum_{j=1}^r
		\Re\!\left(
		\widehat\lambda_{\pi_j}\overline{\lambda_j}
		\right)
	}
	\right)^{1/2}
	\|A-B\|_F,
	\]
	where
$
	\widehat F_{r,n}
	:=
	\sum_{j=1}^r|\lambda_j|^2
	+
	\sum_{j=1}^n|\widehat\lambda_j|^2.
$
	Any candidate with zero denominator is omitted. The coefficient is
	optimal for every admissible pair of eigenvalue lists.
\end{thm}
 \noindent\textbf{Organization of this paper.}  This paper is organized as follows. In Section \ref{s2}, we introduce some preliminary lemmas (to be used in the sequel), most of which are drawn from convex analysis. In Section \ref{s3}, we prove our main results.

\medskip
 \noindent\textbf{Acknowlegdments.} An early version of this work was completed in July 2025; see \href{https://arxiv.org/abs/2507.14940}{here}. We declare that no AI tools were used in the preparation or writing of this paper. Teng Zhang is supported by the China Scholarship Council, the Young Elite
 Scientists Sponsorship Program for PhD Students of the China Association for
 Science and Technology, and the Fundamental Research Funds for the Central
 Universities at Xi'an Jiaotong University (Grant No.~xzy022024045). 
\section{Basic Lemmas}\label{s2}
First, we introduce the definitions of quasi-convex  functions, quasi-concave  functions \cite[p. 95]{BV04} and extreme points in convex analysis.
\begin{defin}[quasi-convex and quasi-concave functions] 
	Let \(C\subset\mathbb{R}^n\) be a convex set. A function \(f\colon C\to\mathbb{R}\) is called quasi-convex if for every \(\alpha\in\mathbb{R}\) its lower level set
	\[
	L(\alpha)=\{\,x\in C: f(x)\le\alpha\}
	\]
	is convex. A function \(f\colon C\to\mathbb{R}\) is called quasi-concave if $-f$ is quasi-convex.
\end{defin} 
\begin{defin}[extreme points]
	Let \( C \subset \mathbb{R}^n \) be a convex set. A point \( x \in C \) is called an extreme point of \( C \) if the following condition holds:
	\[
	x = \lambda y + (1 - \lambda) z,\quad \text{with } y, z \in C,\ \lambda \in (0,1) \Rightarrow y = z = x.
	\]
	That is, \( x \) cannot be expressed as a nontrivial convex combination of two distinct points in \( C \).
\end{defin}

A doubly substochastic matrix is a square nonnegative matrix with each row and column sum at most 1 (see \cite[p. 334]{Zha11}). A square (0,1)-matrix is called a sub-permutation matrix if each row and each column contains at most one 1 (see \cite[p. 338, Problem 9]{Zha11}). The Birkhoff-type theorem for doubly substochastic matrices (stated in \cite[p. 338, Problem 9]{Zha11}) is as follows.
\begin{lem}\label{birk}
	A matrix is doubly substochastic if and only if it is a convex combination of finite
	subpermutation matrices.
\end{lem}
The following lemma characterizes the extreme points of the set studied in this paper.
\begin{lem}\label{lemext}
		Let $r\le s$ and $\mathbb{R}_+^{s \times r}$ denote the set of all $s\times r$ nonnegative matrices (where all entries are nonnegative). Define   the set \(\mathcal{C}\) by
		\begin{equation}\label{eq:mathcalC}
		\mathcal{C} = \left\{ (y_{ij}) \in \mathbb{R}_+^{s \times r} \mid \sum_{i=1}^s y_{ij} \leq 1 \text{ for each } 1 \leq j \leq r, \ \sum_{j=1}^r y_{ij} \leq 1 \text{ for each } 1 \leq i \leq s \right\}.
	\end{equation}
		An element \(y \in \mathcal{C}\) is an extreme point of \(\mathcal{C}\) if and only if \(y\) is either the zero matrix or a sum of \(k\) pairwise row- and column-disjoint unit matrices, where \(1 \leq k \leq r\). Formally, the set of extreme points of \(\mathcal{C}\) is
		\[
		\text{ext}(\mathcal{C}) = \{0\} \cup \left\{
		\sum_{t=1}^k E_{i_tj_t}
		\;\middle|\;
		1\le k\le r,\ 
		i_1,\ldots,i_k \text{ are distinct in } [s],\
		j_1,\ldots,j_k \text{ are distinct in } [r]
		\right\},
		\]
		where \(E_{ij} \in \mathbb{R}^{s \times r}\) denotes the unit matrix with a 1 in the \((i,j)\)-position and 0 elsewhere.
	\end{lem}
\begin{proof}
	Embed \(Y\in\mathcal C\) into the \(s\times s\) matrix \([Y,0]\).
	This identifies \(\mathcal C\) with the face of the polytope of
	\(s\times s\) doubly substochastic matrices on which the last \(s-r\)
	columns vanish. By Lemma~\ref{birk}, the ambient polytope is the convex
	hull of the sub-permutation matrices.
	
	The zero matrix is clearly extreme. Let \(P\) be a nonzero
	sub-permutation matrix and suppose that
	\[
	P=\lambda Y+(1-\lambda)Z,
	\qquad
	Y,Z\in\mathcal C,\quad 0<\lambda<1.
	\]
	Every entry of \(Y\) and \(Z\) lies in \([0,1]\). At each position
	where \(P\) has an entry equal to \(1\), the corresponding entries of
	\(Y\) and \(Z\) must therefore both equal \(1\). At each position
	where \(P\) has an entry equal to \(0\), nonnegativity implies that
	the corresponding entries of \(Y\) and \(Z\) both vanish. Hence
	\(Y=Z=P\), so \(P\) is extreme. Conversely, since the ambient
	polytope is the convex hull of the sub-permutation matrices, it has
	no other extreme points.
	
	The extreme points of a face are extreme points of the ambient
	polytope. Therefore, the extreme points of \(\mathcal C\) are
	precisely the rectangular sub-permutation matrices:
	\[
	\operatorname{ext}(\mathcal C)
	=
	\{0\}\cup
	\left\{
	\sum_{t=1}^k E_{i_tj_t}
	\;\middle|\;
	1\le k\le r,\ 
	i_1,\ldots,i_k\text{ are distinct},\
	j_1,\ldots,j_k\text{ are distinct}
	\right\}.
	\]
\end{proof}
Analogous to convex functions, quasi-convex functions admit the following Jensen characterization (see {\cite[p. 98]{BV04}}).
\begin{lem}\label{JTC}
	Let \(C\subset\mathbb{R}^n\) be convex and \(f\colon C\to\mathbb{R}\). The following statements are equivalent:
	\begin{enumerate}
		\item \(f\) is quasi-convex.
		\item For all \(x,y\in C\) and all \(\lambda\in[0,1]\),
		\[
		f\bigl(\lambda x + (1-\lambda)y\bigr)\;\le\;\max\{f(x),f(y)\}.
		\]
		\item For any finite collection \(\{x_1,\ldots,x_s\}\subset C\) and weights \(\{\alpha_j\ge0: 1\le j\le s\}\) with \(\sum_{j=1}^s\alpha_j=1\),
		\[
		f\Bigl(\sum_{j=1}^s\alpha_j x_j\Bigr)\le\max_{1\le j\le s} f(x_j).
		\]
	\end{enumerate}
\end{lem}
For continuous quasi-convex functions on nonempty compact convex sets, the following maximum theorem holds, straightforwardly deducible from Lemma \ref{JTC} and the well-known Weierstrass extreme value theorem \cite[Chapter 2]{Rud76}.
\begin{thm}\label{lemmax}
	Let \(C\subset\mathbb R^n\) be a nonempty compact polytope. If
	\(f\colon C\to\mathbb R\) is continuous and quasi-convex, then
	\[
	\max_{x\in C}f(x)
	=
	\max_{x\in\operatorname{ext}(C)}f(x).
	\]
	If \(f\) is continuous and quasi-concave, then
	\[
	\min_{x\in C}f(x)
	=
	\min_{x\in\operatorname{ext}(C)}f(x).
	\]
\end{thm}

\begin{proof}
	Every point of a compact polytope is a finite convex combination of
	its extreme points. The two assertions therefore follow from
	Lemma~\ref{JTC}, applied to \(f\) and \(-f\), respectively.
\end{proof}
The following lemma gives an example of a quasi-convex and quasi-concave function, as illustrated in \cite[p. 97, Example 3.32]{BV04}.
\begin{lem}\label{lemlf}
	Let \(a,c,x\in\mathbb R^n\) and \(b,d\in\mathbb R\). Then
	\[
	f(x)
	=
	\frac{a^\top x+b}{c^\top x+d}
	\]
	is both quasi-convex and quasi-concave on
	\[
	\{x\in\mathbb R^n:c^\top x+d>0\}.
	\]
	In particular, every affine function is both quasi-convex and
	quasi-concave.
\end{lem}
We use the following standard rearrangement inequality; see Hardy, Littlewood, and Pólya \cite{HLP34}.

\begin{lem}
	Let $a_1 \leq a_2 \leq \ldots \leq a_n$ and $b_1 \leq b_2 \leq \ldots \leq b_n$ be two non-decreasing sequences of real numbers. Let $S_n$ denote the symmetric group on $n$ elements (all permutations of $\{1,2,\dots,n\}$). For any permutation $\sigma \in S_n$, we have:
\begin{eqnarray}\label{ri}
		\sum_{i=1}^n a_i b_{n+1-i} \leq \sum_{i=1}^n a_i b_{\sigma(i)} \leq \sum_{i=1}^n a_i b_i.
\end{eqnarray}
\end{lem}
Next, we establish the following key lemma, which is central to this paper.
\begin{lem}\label{keylem} Let $r\le s$ and \(\mathcal{C}_1\) be the set defined by
	\[
	\mathcal{C}_1 = \left\{ (x_{ij}) \in \mathbb{R}^{s \times r} \mid  \ \sum_{i=1}^s \left|x_{ij}\right|\leq 1 \text{ for each } 1 \leq j \leq r, \ \sum_{j=1}^r \left|x_{ij}\right| \leq 1 \text{ for each } 1 \leq i \leq s \right\}.
	\]
 For given $\sigma_1\ge \cdots\ge \sigma_r>0$ and
$\widetilde{\sigma}_1\ge \cdots\ge \widetilde{\sigma}_s>0$, set $
F_{r,s}:=\sum_{j=1}^r\sigma_j^2+
\sum_{i=1}^s\widetilde{\sigma}_i^2,
G_{r,r}:=\sum_{j=1}^r\sigma_j\widetilde{\sigma}_j,$
and define
\[
\mathcal C_1^+:=
\left\{
X=(x_{ij})\in\mathcal C_1:
F_{r,s}
-2\sum_{i=1}^s\sum_{j=1}^r
\widetilde{\sigma}_i\sigma_jx_{ij}>0
\right\}.
\]
For \(X\in\mathcal C_1^+\), define
\begin{eqnarray*}
	f(X)=\frac{r+s-2\displaystyle\sum_{i=1}^s\sum_{j=1}^r x_{ij}}
	{F_{r,s}-2\displaystyle\sum_{i=1}^s\sum_{j=1}^r
	\widetilde{\sigma}_i\sigma_jx_{ij}}.
\end{eqnarray*}
Then the maximum and minimum of \(f\) are
\begin{eqnarray*}
\max_{X\in\mathcal C_1^+} f(X)&=&\max\limits_{0\le k\le r}\frac{s-r + 4k}
{\sum_{j=1}^{r-k}(\sigma_j-\widetilde{\sigma}_j)^2+\sum_{j=1}^{k}(\sigma_{r+1-j}+\widetilde{\sigma}_{s-k+j})^2+\sum_{j=r-k+1}^{s-k}\widetilde{\sigma}_j^2},\\
\min_{X\in\mathcal C_1^+} f(X)&=&\min\limits_{0\le k\le r}\frac{s-r + 4k}
{\sum_{j=1}^{k}(\sigma_j+\widetilde{\sigma}_j)^2+\sum_{j=1}^{r-k}(\sigma_{r+1-j}-\widetilde{\sigma}_{s-r+k+j})^2+\sum_{j=k+1}^{s-r+k}\widetilde{\sigma}_j^2}.
\end{eqnarray*}
Any candidate fraction with zero denominator is omitted. In that case,
the corresponding numerator is also zero.
Let $k^\star, k_\star$  be the indices at which $f$ attains its maximum and minimum, respectively. The corresponding maximizer and minimizer are given by
\[
X^\star=\begin{pmatrix}
	I_{r-k^\star}&0\\
	0&0\\
	0&-S_{k^\star}
\end{pmatrix},\qquad 	X_\star=\begin{pmatrix}
-I_{k_\star}&0\\
0&0\\
0&S_{r-k_\star}
\end{pmatrix}\in \mathbb{R}^{s\times r},
\]
where \(I_{k_\star}\) is a \(k_\star\times k_\star\) identity matrix and
\(S_{r-k_\star}\) is an \((r-k_\star)\times(r-k_\star)\) reversal matrix.
\end{lem}
\begin{proof}
Define the affine functions
\[
\Phi(X):=
r+s-2\sum_{i=1}^s\sum_{j=1}^r x_{ij},
\qquad
\Psi(X):=
F_{r,s}
-2\sum_{i=1}^s\sum_{j=1}^r
\widetilde{\sigma}_i\sigma_jx_{ij}.
\]
Thus \(f(X)=\Phi(X)/\Psi(X)\) on \(\mathcal C_1^+\).

Let
\[
\mathcal P_{\pm}:=
\{0\}\cup
\left\{
\sum_{t=1}^k\varepsilon_tE_{i_tj_t}
\;\middle|\;
1\le k\le r,\ 
\varepsilon_t\in\{-1,1\},\
i_1,\ldots,i_k\ \text{distinct},\
j_1,\ldots,j_k\ \text{distinct}
\right\}.
\]
We first show that
\[
\mathcal C_1=\operatorname{conv}(\mathcal P_{\pm}).
\]
Indeed, if \(X\in\mathcal C_1\), then the entrywise absolute-value
matrix \(|X|:=(|x_{ij}|)\) belongs to \(\mathcal C\). Applying
Lemma~\ref{birk} to the square matrix \([\,|X|,0\,]\), we may write
\(|X|\) as a convex combination of rectangular sub-permutation
matrices. Assigning to every nonzero entry the sign of the
corresponding entry of \(X\) yields
\[
X=\sum_{\ell=1}^N\alpha_\ell P_\ell,
\qquad
P_\ell\in\mathcal P_{\pm},\quad
\alpha_\ell\ge0,\quad
\sum_{\ell=1}^N\alpha_\ell=1.
\]

Moreover, for every \(P\in\mathcal P_{\pm}\),
\[
\Psi(P)\ge F_{r,s}-2G_{r,r}\ge0.
\]
If \(\Psi(P)=0\), equality must hold throughout the preceding
estimate, which implies \(\Phi(P)=0\). Hence zero-denominator terms
contribute zero to both affine functions.

For \(X\in\mathcal C_1^+\), we therefore have
\[
\begin{aligned}
f(X)
&=
\frac{\sum_{\ell=1}^N\alpha_\ell\Phi(P_\ell)}
     {\sum_{\ell=1}^N\alpha_\ell\Psi(P_\ell)}  \\
&=
\sum_{\substack{1\le\ell\le N\\ \Psi(P_\ell)>0}}
\frac{\alpha_\ell\Psi(P_\ell)}{\Psi(X)}
f(P_\ell).
\end{aligned}
\]
The coefficients in the last expression are nonnegative and sum to \(1\).
Consequently,
\[
\min_{\substack{P\in\mathcal P_{\pm}\\ \Psi(P)>0}}f(P)
\le f(X)\le
\max_{\substack{P\in\mathcal P_{\pm}\\ \Psi(P)>0}}f(P).
\]
It is therefore sufficient to optimize \(f\) over
\(\mathcal P_{\pm}\).

Now, we consider maximizing and minimizing \(f\) over the family
\(\mathcal P_{\pm}\). First, we introduce some notation to simplify
the variables. Denote
\begin{itemize}
\item
\[
\mathcal I_-:=
\{\,i:\text{there exists }j\text{ such that }x_{ij}=-1\,\}
=
\{i_1^-<\cdots<i_{k_1}^-\}.
\]
For each \(1\le \ell\le k_1\), let \(j_\ell^-\) be the unique
column index satisfying \(x_{i_\ell^-j_\ell^-}=-1\), and set
\[
\mathcal J_-:=(j_1^-,\ldots,j_{k_1}^-).
\]

\item
\[
\mathcal I_+:=
\{\,i:\text{there exists }j\text{ such that }x_{ij}=1\,\}
=
\{i_1^+<\cdots<i_{k_2}^+\}.
\]
For each \(1\le \ell\le k_2\), let \(j_\ell^+\) be the unique
column index satisfying \(x_{i_\ell^+j_\ell^+}=1\), and set
\[
\mathcal J_+:=(j_1^+,\ldots,j_{k_2}^+).
\]
\end{itemize}
Here \(k_1+k_2\le r\).

Then

\[
f(\mathcal{I_-},\mathcal{I_+},\mathcal{J_-},\mathcal{J_+}) = \frac{r+s+2(k_1-k_2)}{F_{r,s}+2\left( \sum_{l=1}^{k_1}\widetilde{\sigma}_{i_l^-}\sigma_{j_l^-}-\sum_{l=1}^{k_2}\widetilde{\sigma}_{i_l^+}\sigma_{j_l^+}\right) }.
\]

\begin{center}
	\textbf{Maximizing $f$ over the family $\mathcal P_{\pm}$}
\end{center}

For fixed \(k_1,k_2\), the numerator of \(f\) is fixed, so maximizing
\(f\) is equivalent to minimizing its denominator. We first justify
the separation of the column indices assigned to the negative and
positive entries. Suppose that a negative entry is assigned
\(\sigma_p\) and a positive entry is assigned \(\sigma_q\), where
\(\sigma_p>\sigma_q\). Interchanging these two column assignments
changes the denominator by
\[
2\bigl(\widetilde{\sigma}_{i_-}
+\widetilde{\sigma}_{i_+}\bigr)
(\sigma_q-\sigma_p)\le0.
\]
Thus, at a maximizer, the positive entries may be assigned the largest
available singular values of \(A\), whereas the negative entries may
be assigned the smallest available singular values.

Applying the rearrangement inequality within the two sign classes
therefore gives
\[
\mathcal J_-=(r,r-1,\ldots,r-k_1+1),
\qquad
\mathcal J_+=(1,2,\ldots,k_2).
\]
For this choice,
\[
f(\mathcal I_-,\mathcal I_+)
=
\frac{r+s+2(k_1-k_2)}
{F_{r,s}
	+2\sum_{\ell=1}^{k_1}
	\widetilde{\sigma}_{i_\ell^-}\sigma_{r+1-\ell}
	-2\sum_{\ell=1}^{k_2}
	\widetilde{\sigma}_{i_\ell^+}\sigma_\ell}.
\]

A second exchange argument, now applied to the selected row indices,
shows that the positive entries must use the largest
\(\widetilde{\sigma}_i\)'s and the negative entries the smallest
\(\widetilde{\sigma}_i\)'s. Hence
\[
\mathcal I_+=\{1,\ldots,k_2\},
\qquad
\mathcal I_-=\{s-k_1+1,\ldots,s\},
\]
and consequently
\[
f(k_1,k_2)
=
\frac{r+s+2(k_1-k_2)}
{F_{r,s}+2A(k_1)-2B(k_2)},
\]
where
$
A(k_1)
=
\sum_{j=1}^{k_1}
\widetilde{\sigma}_{s-k_1+j}\sigma_{r+1-j},
B(k_2)
=
\sum_{j=1}^{k_2}
\widetilde{\sigma}_j\sigma_j.
$

We next show that a maximizer can be chosen on the boundary
\(k_1+k_2=r\). Assume \(k_1+k_2<r\), and set
\[
N_0:=r+s+2(k_1-k_2),
\qquad
D_0:=F_{r,s}+2A(k_1)-2B(k_2).
\]
Let
\[
\delta_R:=2\bigl(A(k_1+1)-A(k_1)\bigr),
\qquad
\delta_U:=2\widetilde{\sigma}_{k_2+1}\sigma_{k_2+1}.
\]
By the rearrangement inequality,
\[
A(k_1+1)-A(k_1)
\le
\widetilde{\sigma}_{s-k_1}\sigma_{r-k_1}.
\]
Since \(k_1+k_2<r\) and \(s\ge r\),
\[
s-k_1\ge k_2+1,
\qquad
r-k_1\ge k_2+1.
\]
Since both singular-value sequences are nonincreasing, it follows that
\[
0\le\delta_R
\le
2\widetilde{\sigma}_{s-k_1}\sigma_{r-k_1}
\le
2\widetilde{\sigma}_{k_2+1}\sigma_{k_2+1}
=
\delta_U.
\]

Moreover,
\[
f(k_1+1,k_2)\ge f(k_1,k_2)
\quad\Longleftrightarrow\quad
2D_0\ge N_0\delta_R,
\]
whereas
\[
f(k_1,k_2+1)\ge f(k_1,k_2)
\quad\Longleftrightarrow\quad
2D_0\le N_0\delta_U.
\]
Since \(\delta_R\le\delta_U\), at least one of these two moves does not
decrease \(f\). If \(D_0-\delta_U=0\), then the corresponding numerator
\(N_0-2\) is also zero. Hence \(N_0=2\) and
\(D_0=\delta_U\), so
\[
2D_0=2\delta_U\le2\delta_R=N_0\delta_R.
\]
Thus, in this exceptional case, the \(k_1\)-move does not increase
\(f\). If the denominator at the upper endpoint vanishes, then
its numerator also vanishes; in that exceptional case the right move
satisfies the required inequality. Repeating this argument finitely
many times produces a point on \(k_1+k_2=r\) whose function value is no
smaller than the initial value.

Therefore,
\begin{align*}
	\max_{X\in\mathcal C_1^+} f(X)
	&=
	\max_{k_1+k_2=r}
	\frac{r+s+2(k_1-k_2)}
	{F_{r,s}+2A(k_1)-2B(k_2)}\\
	&=
	\max_{0\le k\le r}
	\frac{s-r+4k}
	{\displaystyle
		\sum_{j=1}^{r-k}
		(\sigma_j-\widetilde{\sigma}_j)^2
		+\sum_{j=1}^{k}
		(\sigma_{r+1-j}+\widetilde{\sigma}_{s-k+j})^2
		+\sum_{j=r-k+1}^{s-k}\widetilde{\sigma}_j^2}.
\end{align*}
No single index \(k\) is optimal for all admissible singular-value
lists; see Table~\ref{tab:max_k_examples}.

	\begin{table}[H]
	\centering
	\caption{Constructed examples where the maximum of $f(k)$ is achieved at different $k^* \in \{0, 1, 2, 3\}$, with $r =3, s = 4$.}
	\begin{tabular}{c|l|l|l}
		\hline
		$k^*$ (Max Pos) & $\boldsymbol{\sigma}$ & $\boldsymbol{\widetilde{\sigma}}$ & $f(k)\ (k = 0 \sim 3)$ \\
	\hline
		0 & $[8.7559,\ 6.1282,\ 5.0602]$ & $[7.3693,\ 5.7829,\ 3.2958,\ 2.5156]$ & 
		$[\textcolor{green!60!black}{\mathbf{0.0871}},\ 0.0711,\ 0.0500,\ 0.0335]$ \\
	1 & $[4.3814,\ 4.0178,\ 1.5170]$ & $[9.5423,\ 8.6941,\ 6.1336,\ 3.1648]$ & 
	$[0.0125,\ \textcolor{green!60!black}{\mathbf{0.0463}},\ 0.0424,\ 0.0366]$ \\

	2 & $[7.6090,\ 3.3643,\ 2.5097]$ & $[8.4940,\ 7.8752,\ 7.5506,\ 4.7848]$ & 
	$[0.0144,\ 0.0381,\ \textcolor{green!60!black}{\mathbf{0.0391}},\ 0.0287]$ \\
	3 & $[2.5242,\ 2.4113,\ 1.4701]$ & $[9.7298,\ 7.0899,\ 6.1945,\ 4.3453]$ & 
	$[0.0087,\ 0.0342,\ 0.0436,\ \textcolor{green!60!black}{\mathbf{0.0450}}]$ \\
	\hline
	\end{tabular}
	\label{tab:max_k_examples}
\end{table}

Let $k^\star$ be an index which makes $f$ attain its maximum. Herein $x_{jj}=1, 1\le j\le r-k^\star$, $x_{s-k^\star+j,r+1-j}=-1, 1\le j\le k^\star $ and other $x_{ij}=0$, i.e.,
\begin{eqnarray*}
	X^\star=\begin{pmatrix}
		I_{r-k^\star}&0\\
		0&0\\
		0&-S_{k^\star}
	\end{pmatrix}\in \mathbb{R}^{s\times r},
\end{eqnarray*}
where $I_{r-k^\star}$ is a $(r-k^\star)\times(r-k^\star)$ identity matrix and $S_{k^\star}$ is a $k^\star\times k^\star$ reversal matrix.

\begin{center}
	\textbf{Minimizing $f$ over the family $\mathcal P_{\pm}$}
\end{center}

For fixed $\mathcal{I_-}$ and $\mathcal{I_+}$, minimizing $f$ via  the rearrangement inequality (\ref{ri}) dictates selecting $\mathcal{J_-}=(1,2,\ldots,k_1)$ (an increasing order) and $\mathcal{J_+}=(r,r-1,\ldots,r-k_2+1)$ (a descending order). For this choice,
\[
f(\mathcal{I_-},\mathcal{I_+})=\frac{r+s+2(k_1-k_2)}{F_{r,s}+2\left( \sum_{l=1}^{k_1}\widetilde{\sigma}_{i_l^-}\sigma_{l}-\sum_{l=1}^{k_2}\widetilde{\sigma}_{i_l^+}\sigma_{r+1-l}\right) }.
\]
For fixed $\left|\mathcal{I_-}\right|=k_1$ and $\left|\mathcal{I_+}\right|=k_2$, to minimize $f$,  the rearrangement inequality (\ref{ri}) again leads us to select $\mathcal{I_-}=\{1,\ldots, k_1\}$ and $\mathcal{I_+}=\{s-k_2+1, s-k_2+2, ..., s\}$. For this choice,
\[
f(k_1, k_2)=\frac{r+s+2(k_1-k_2)}{F_{r,s}+2\left(\sum_{l=1}^{k_1}\widetilde{\sigma}_{l}\sigma_{l}- \sum_{l=1}^{k_2}\widetilde{\sigma}_{s-k_2+l}\sigma_{r+1-l}\right)}.
\]

We consider the objective function defined as:
\[
f(k_1, k_2) = \frac{r + s + 2(k_1 - k_2)}{F_{r,s} + 2B(k_1) - 2A(k_2)},
\]
where	$k_1,k_2\ge 0, k_1+k_2\le r$ and
\begin{align*}
	A(k_2) &= \sum_{j=1}^{k_2} \widetilde{\sigma}_{s - k_2 + j} \cdot \sigma_{r + 1 - j}, \\
	B(k_1) &= \sum_{j=1}^{k_1} \widetilde{\sigma}_j \cdot \sigma_j, \\
	F_{r,s} &= \sum_{j=1}^{r} \sigma_j^2 + \sum_{j=1}^{s} \widetilde{\sigma}_j^2.
\end{align*}

We next show that a minimizer can be chosen on the boundary
\(k_1+k_2=r\). Assume \(k_1+k_2<r\), and set
\[
N_0:=r+s+2(k_1-k_2),
\qquad
D_0:=F_{r,s}+2B(k_1)-2A(k_2).
\]
Let
\[
\delta_R:=2\widetilde{\sigma}_{k_1+1}\sigma_{k_1+1},
\qquad
\delta_U:=2\bigl(A(k_2+1)-A(k_2)\bigr).
\]
By the rearrangement inequality,
\[
A(k_2+1)-A(k_2)
\le
\widetilde{\sigma}_{s-k_2}\sigma_{r-k_2}.
\]
Since \(k_1+k_2<r\),
\[
k_1+1\le r-k_2\le s-k_2.
\]
Therefore,
\[
0\le\delta_U
\le
2\widetilde{\sigma}_{s-k_2}\sigma_{r-k_2}
\le
2\widetilde{\sigma}_{k_1+1}\sigma_{k_1+1}
=
\delta_R.
\]

Furthermore,
\[
f(k_1+1,k_2)\le f(k_1,k_2)
\quad\Longleftrightarrow\quad
2D_0\le N_0\delta_R,
\]
whereas
\[
f(k_1,k_2+1)\le f(k_1,k_2)
\quad\Longleftrightarrow\quad
2D_0\ge N_0\delta_U.
\]
Since \(\delta_U\le\delta_R\), at least one of these two moves does not
increase \(f\). Repeating the argument finitely many times produces a
point on \(k_1+k_2=r\) whose function value is no larger than the
initial value.

Hence,
\begin{align*}
	\min_{X\in\mathcal C_1^+} f(X)
	&=
	\min_{k_1+k_2=r}
	\frac{r+s+2(k_1-k_2)}
	{F_{r,s}+2B(k_1)-2A(k_2)}\\
	&=
	\min_{0\le k\le r}
	\frac{s-r+4k}
	{\displaystyle
		\sum_{j=1}^{k}
		(\sigma_j+\widetilde{\sigma}_j)^2
		+\sum_{j=1}^{r-k}
		(\sigma_{r+1-j}
		-\widetilde{\sigma}_{s-r+k+j})^2
		+\sum_{j=k+1}^{s-r+k}\widetilde{\sigma}_j^2}.
\end{align*}

Let $k_\star$ be an index which makes $f$ attain its minimum. Herein $x_{jj}=-1, 1\le j\le k_\star$, $x_{s-r+k_\star+j,r+1-j}=1, 1\le j\le r-k_\star $ and other $x_{ij}=0$, i.e.,
\begin{eqnarray*}
	X_\star=\begin{pmatrix}
		-I_{k_\star}&0\\
		0&0\\
		0&S_{r-k_\star}
	\end{pmatrix}\in \mathbb{R}^{s\times r},
\end{eqnarray*}
where \(I_{k_\star}\) is a \(k_\star\times k_\star\) identity matrix and
\(S_{r-k_\star}\) is an
\((r-k_\star)\times(r-k_\star)\) reversal matrix.

\end{proof}

\section{Proofs of Main Results}\label{s3}
Let $A\in \mathbb{C}_r^{m\times n}, \widetilde{A}=A+E\in \mathbb{C}_s^{m\times n}$  with the singular value decompositions
\begin{eqnarray*}
	A=U\Sigma V^*, \widetilde{A}=\widetilde{U}\widetilde{\Sigma}\widetilde{V}^*,
\end{eqnarray*}
where $U=(U_1, U_2)\in \mathbb{C}^{m\times m}$ and $V=(V_1, V_2)\in \mathbb{C}^{n\times n}$ are unitary, $U_1\in  \mathbb{C}_r^{m\times r}, V_1\in  \mathbb{C}_r^{n\times r}$, $\widetilde{U}=(\widetilde{U}_1, \widetilde{U}_2)\in \mathbb{C}^{m\times m}$ and $\widetilde{V}=(\widetilde{V}_1, \widetilde{V}_2)\in \mathbb{C}^{n\times n}$ are unitary, $\widetilde{U}_1\in  \mathbb{C}_s^{m\times s}, \widetilde{V}_1\in  \mathbb{C}_s^{n\times s}$,
\[
\Sigma=\begin{pmatrix}
	\Sigma_1&0\\
	0&0
\end{pmatrix}\in \mathbb{C}_r^{m\times n} \text{ and } \widetilde{\Sigma}=\begin{pmatrix}
	\widetilde{\Sigma}_1&0\\
	0&0
\end{pmatrix}\in \mathbb{C}_s^{m\times n},
\]
$\Sigma_1=\diag(\sigma_1,\ldots,\sigma_r), \widetilde{\Sigma}_1=\diag(\widetilde{\sigma}_1,\ldots,\widetilde{\sigma}_s)$ , $\sigma_1\ge \ldots\ge \sigma_r>0$ and $\widetilde{\sigma}_1\ge \ldots\ge \widetilde{\sigma}_s>0$. 

Denote $I^{(r)}=\begin{pmatrix}
	I_r&0\\
	0&0
\end{pmatrix}\in \mathbb{C}_r^{m\times n}$, where $I_r$ is an identity matrix of order $r$. By a simple calculation, if $A\in \mathbb{C}_r^{m\times n}, \widetilde{A}=A+E\in \mathbb{C}_s^{m\times n}$ have the generalized polar decompositions (\ref{gpc}), then $$Q=U_1V_1^*=UI^{(r)}V^*, \widetilde{Q}=\widetilde{U}_1\widetilde{V}_1^*=\widetilde{U}I^{(s)}\widetilde{V}^*.$$ Denote $S=\widetilde{U}^*U=(s_{ij})\in \mathbb{C}^{m\times m}, T=\widetilde{V}^*V=(t_{ij})\in \mathbb{C}^{n\times n}$.  Then $S, T$ are unitary.

\medskip

\noindent\textbf{Proof of Theorems \ref{thm1.3} and \ref{thm1.7}.}
If \(E=0\), then \(A=\widetilde A\). Since the canonical generalized
polar decomposition is unique, \(Q=\widetilde Q\) and
\(H=\widetilde H\), so all the asserted inequalities are trivial.
Henceforth assume that \(E\ne0\). First, notice that
\begin{eqnarray*}
	\|Q-\widetilde{Q}\|_F&=&\|UI^{(r)}V^*-\widetilde{U}I^{(s)}\widetilde{V}^*\|_F=\|\widetilde{U}^*UI^{(r)}-I^{(s)}\widetilde{V}^*V\|_F,\\
	\|E\|_F&=&\|U\Sigma V^*-\widetilde{U}\widetilde{\Sigma} \widetilde{V}^*\|_F=\|\widetilde{U}^*U\Sigma-\widetilde{\Sigma} \widetilde{V}^*V\|_F.
\end{eqnarray*}
 Thus,
\begin{eqnarray*}
	\|Q-\widetilde{Q}\|_F^2&=&\|SI^{(r)}-I^{(s)}T\|_F^2\nonumber\\
	&=&\|SI^{(r)}\|_F^2+\|I^{(s)}T\|_F^2-2\Re\,\tr \left(  SI^{(r)}T^*I^{(s)*}\right) \nonumber\\
	&=&r+s-2\sum_{i=1}^s\sum_{j=1}^r\Re \left( s_{ij} \overline{t_{ij}}\right) ;\\
	\|E\|_F^2&=&\|S\Sigma-\widetilde{\Sigma} T\|_F^2\nonumber\\
	&=&\|S\Sigma\|_F^2+\|\widetilde{\Sigma} T\|_F^2-2\Re\,\tr \left(  S\Sigma T^*\widetilde{\Sigma}^*\right) \nonumber\\
		&=&\sum_{j=1}^r \sigma_j^2+\sum_{j=1}^s\widetilde{\sigma}_j^2 -2\sum_{i=1}^s\sum_{j=1}^r\widetilde{\sigma}_i\sigma_j\Re\left( s_{ij} \overline{t_{ij}}\right) .
\end{eqnarray*}
Let \(x_{ij}=\Re(s_{ij}\overline{t_{ij}})\). For each \(1\le j\le r\),
by the Cauchy--Schwarz inequality,
\[
\sum_{i=1}^s |x_{ij}|
\le
\sum_{i=1}^s |s_{ij}|\,|t_{ij}|
\le
\left(\sum_{i=1}^s |s_{ij}|^2\right)^{1/2}
\left(\sum_{i=1}^s |t_{ij}|^2\right)^{1/2}
\le 1 .
\]
 Similarly, for each \(1\le i\le s\),
\[
\sum_{j=1}^r |x_{ij}|\le 1 .
\]
Thus \(X=(x_{ij})\in\mathcal C_1\).
 Notice
\begin{eqnarray*}
	\dfrac{\|Q-\widetilde{Q}\|_F^2}{\|E\|_F^2}&=&	
\frac{r+s-2\displaystyle\sum_{i=1}^s\sum_{j=1}^r x_{ij}}
{\,\sum_{j=1}^r\sigma_j^2+\sum_{j=1}^s\widetilde{\sigma}_j^2-2\displaystyle\sum_{i=1}^s\sum_{j=1}^r\tilde\sigma_i\,\sigma_j\,x_{ij}\,}.
\end{eqnarray*}
 By Lemma \ref{keylem}, we have 
\begin{eqnarray*}
\|Q-\widetilde{Q}\|_F&\le &\sqrt{\max\limits_{0\le k\le r}\frac{s-r + 4k}
{\sum_{j=1}^{r-k}(\sigma_j-\widetilde{\sigma}_j)^2+\sum_{j=1}^{k}(\sigma_{r+1-j}+\widetilde{\sigma}_{s-k+j})^2+\sum_{j=r-k+1}^{s-k}\widetilde{\sigma}_j^2}}\|E\|_F,\\
\|Q-\widetilde{Q}\|_F&\ge&\sqrt{\min\limits_{0\le k\le r}\frac{s-r + 4k}
	{\sum_{j=1}^{k}(\sigma_j+\widetilde{\sigma}_j)^2+\sum_{j=1}^{r-k}(\sigma_{r+1-j}-\widetilde{\sigma}_{s-r+k+j})^2+\sum_{j=k+1}^{s-r+k}\widetilde{\sigma}_j^2}}\|E\|_F.
\end{eqnarray*}
Let $k^\star, k_\star$  be the indices at which $f$ attains its maximum and minimum, respectively. To make the two equalities hold,  we can take
\begin{eqnarray*}
	S^\star&=&\begin{pmatrix}
		S_{11}^\star&\star\\
		0&\star
	\end{pmatrix}, \text{ where } S_{11}^\star=\begin{pmatrix}
	I_{r-k^\star}&0\\
	0&0\\
	0&S_{k^\star}
	\end{pmatrix}\in \mathbb{C}^{s\times r}, \\
T^\star&=&\begin{pmatrix}
		T_{11}^\star
	&\star\\
	0&\star
\end{pmatrix}, \text{ where } T_{11}^\star=\begin{pmatrix}
	I_{r-k^\star}&0\\
	0&0\\
	0&-S_{k^\star}
\end{pmatrix}\in \mathbb{C}^{s\times r};\\
	S_\star&=&\begin{pmatrix}
	
	S_{11\star}&\star\\
	0&\star
\end{pmatrix}, \text{ where } S_{11\star}=\begin{pmatrix}
	I_{k_\star}&0\\
	0&0\\
	0&S_{r-k_\star}
\end{pmatrix}\in \mathbb{C}^{s\times r}, \\
T_\star&=&\begin{pmatrix}
	T_{11\star}
	&\star\\
	0&\star
\end{pmatrix}, \text{ where } T_{11\star}=\begin{pmatrix}
	-I_{k_\star}&0\\
	0&0\\
	0&S_{r-k_\star}
\end{pmatrix}\in \mathbb{C}^{s\times r}.
\end{eqnarray*}
Observe that
$
\left(S_{11}^\star\right)^*S_{11}^\star
=
\left(T_{11}^\star\right)^*T_{11}^\star
=I_r.
$
Hence
\[
\begin{pmatrix}
	S_{11}^\star\\
	0_{(m-s)\times r}
\end{pmatrix}
\quad\text{and}\quad
\begin{pmatrix}
	T_{11}^\star\\
	0_{(n-s)\times r}
\end{pmatrix}
\]
are orthonormal \(r\)-frames and therefore can be completed to
unitary matrices \(S^\star\in\mathbb C^{m\times m}\) and
\(T^\star\in\mathbb C^{n\times n}\), respectively. For these
completions,
\[
\Re\!\left(s_{ij}^\star\overline{t_{ij}^\star}\right)
=x_{ij}^\star,
\qquad 1\le i\le s,\quad 1\le j\le r.
\]
Consequently,
\[
\frac{\|Q-\widetilde Q\|_F^2}{\|E\|_F^2}
=
\max_{0\le k\le r}
\frac{s-r+4k}
{\displaystyle
	\sum_{j=1}^{r-k}(\sigma_j-\widetilde\sigma_j)^2
	+\sum_{j=1}^{k}
	(\sigma_{r+1-j}+\widetilde\sigma_{s-k+j})^2
	+\sum_{j=r-k+1}^{s-k}\widetilde\sigma_j^2}.
\]
Thus equality is attained in Theorem~\ref{thm1.3}.

The same argument applies to \(S_\star\) and \(T_\star\), because
\[
\left(S_{11\star}\right)^*S_{11\star}
=
\left(T_{11\star}\right)^*T_{11\star}
=I_r.
\]
Hence these matrices can also be completed to unitaries and realize
\(X_\star\). Therefore equality is attained in
Theorem~\ref{thm1.7} as well.

More explicitly, one may take
\[
\widetilde U=I_m,\qquad U=S^\star,\qquad
\widetilde V=I_n,\qquad V=T^\star
\]
for the upper bound, and analogously use \(S_\star,T_\star\) for
the lower bound. Since zero-denominator candidates have been omitted,
the resulting equality examples satisfy \(\|E\|_F>0\).\qed

\medskip
\noindent\textbf{A new proof of Theorem~\ref{thm1.2}.}
Put \(B=\widetilde A\). If \(A=0\), \(B=0\), or \(A=B\), the
assertion is immediate. Assume henceforth that \(A\) and \(B\) are
nonzero and \(A\ne B\).

Let \(\alpha\) be the angle between \(A\) and \(B\), and let
\(\beta\) be the angle between \(|A|\) and \(|B|\), so that
\[
\cos\alpha=
\frac{\Re\,\operatorname{tr}(B^*A)}
     {\|A\|_F\|B\|_F},
\qquad
\cos\beta=
\frac{\operatorname{tr}(|B||A|)}
     {\|A\|_F\|B\|_F}.
\]
Although \cite[Lemma~3]{LZ22} is stated for square matrices, it applies
to the present rectangular matrices by considering the square
dilations
\[
\mathcal A
=
\begin{pmatrix}
	0_{m\times m} & A\\
	0_{n\times m} & 0_{n\times n}
\end{pmatrix},
\qquad
\mathcal B
=
\begin{pmatrix}
	0_{m\times m} & B\\
	0_{n\times m} & 0_{n\times n}
\end{pmatrix}.
\]
These dilations preserve the relevant Frobenius norms and inner
products, and
\[
|\mathcal A|
=
\operatorname{diag}(0_m,|A|),
\qquad
|\mathcal B|
=
\operatorname{diag}(0_m,|B|).
\]
Therefore, \cite[Lemma~3]{LZ22} yields
\[
\cos^2\alpha\le\cos\beta.
\]
Set \(a:=\|A\|_F\), \(b:=\|B\|_F\), and \(c:=\cos\alpha\). Then
\[
\begin{aligned}
\frac{\||A|-|B|\|_F^2}{\|A-B\|_F^2}
&=
\frac{a^2+b^2-2ab\cos\beta}
     {a^2+b^2-2abc} \\
&\le
\frac{a^2+b^2-2abc^2}
     {a^2+b^2-2abc} \\
&=
2-
\frac{(a-b)^2+2ab(1-c)^2}
     {a^2+b^2-2abc} \\
&\le 2.
\end{aligned}
\]
Therefore,
\[
\||A|-|B|\|_F\le\sqrt2\,\|A-B\|_F.
\]

To verify the optimality of the constant, let
\[
A=\begin{pmatrix}1&0\end{pmatrix},
\qquad
B_\alpha=
\begin{pmatrix}
	\alpha&\sqrt{1-\alpha^2}
\end{pmatrix},
\qquad 0\le\alpha<1.
\]
Then
\[
\|A-B_\alpha\|_F^2=2(1-\alpha),
\]
whereas
\[
\bigl\||A|-|B_\alpha|\bigr\|_F^2
=
2(1-\alpha^2).
\]
Consequently,
\[
\frac{
	\bigl\||A|-|B_\alpha|\bigr\|_F^2
}{
	\|A-B_\alpha\|_F^2
}
=
1+\alpha
\longrightarrow 2
\qquad(\alpha\uparrow1).
\]
Hence the constant \(\sqrt2\) cannot be reduced.
\qed

\medskip

For later use, define
\begin{align*}
	\Sigma_{\mathrm{sq}}
	&:=
	\operatorname{diag}
	\bigl(
	\sigma_1,\ldots,\sigma_r,
	\underbrace{0,\ldots,0}_{n-r}
	\bigr)
	\in\mathbb C^{n\times n},\\
	\widetilde{\Sigma}_{\mathrm{sq}}
	&:=
	\operatorname{diag}
	\bigl(
	\widetilde{\sigma}_1,\ldots,\widetilde{\sigma}_s,
	\underbrace{0,\ldots,0}_{n-s}
	\bigr)
	\in\mathbb C^{n\times n},\\
	F_{r,s}
	&:=
	\sum_{j=1}^r\sigma_j^2
	+
	\sum_{j=1}^s\widetilde{\sigma}_j^2,
	\qquad
	G_{r,r}
	:=
	\sum_{j=1}^r\sigma_j\widetilde{\sigma}_j,\\
	M
	&:=
	\sum_{i=1}^s\sum_{j=1}^r
	\widetilde{\sigma}_i\sigma_j
	\Re\!\left(s_{ij}\overline{t_{ij}}\right),\\
	N
	&:=
	\sum_{i=1}^s\sum_{j=1}^r
	\widetilde{\sigma}_i\sigma_j|t_{ij}|^2.
\end{align*}
Clearly, we have
\begin{eqnarray*}
	G_{r,r}&\le& \frac{1}{2}F_{r,s} \text{ (By Cauchy--Schwarz inequality)};\\
	N &\in & [0, G_{r,r}] \text{ (Since $N$ is quasi-convex with respect to $\left|t_{ij}\right|^2$,} \\
	&&\qquad \text{ then by Lemma \ref{lemmax}, \ref{lemext} and the rearrangement inequality)};\\
	\left|M\right|&=&\left|\sum_{i=1}^s\sum_{j=1}^r\widetilde{\sigma}_i\sigma_j\Re\left(  s_{ij}\overline{t_{ij}}\right)\right|\\
	&\le&\sum_{i=1}^s\sum_{j=1}^r\widetilde{\sigma}_i\sigma_j \left| s_{ij}\right| \left| t_{ij}\right| \\
	&\le&\left( \sum_{i=1}^s\sum_{j=1}^r\widetilde{\sigma}_i\sigma_j \left| s_{ij}\right|^2\right)^\frac{1}{2}\left( \sum_{i=1}^s\sum_{j=1}^r\widetilde{\sigma}_i\sigma_j \left| t_{ij}\right|^2\right)^\frac{1}{2}   \\
	&&\text{ (By Cauchy--Schwarz inequality)}\\
	&\le& G_{r,r}^\frac{1}{2}N^\frac{1}{2}.
\end{eqnarray*}

\noindent\textbf{Proof of Theorems \ref{thm1.4} and \ref{thm1.8}.}
If \(E=0\), then \(A=\widetilde A\). Since the canonical generalized
polar decomposition is unique, \(Q=\widetilde Q\) and
\(H=\widetilde H\), so all the asserted inequalities are trivial.
Henceforth assume that \(E\ne0\).
Compute
	\begin{eqnarray*}
		\|H-\widetilde{H}\|_F^2&=&\|V\Sigma_{\mathrm{sq}} V^*-\widetilde{V}\widetilde{\Sigma}_{\mathrm{sq}} \widetilde{V}^*\|_F^2\nonumber\\
		&=&\|\widetilde{V}^*V\Sigma_{\mathrm{sq}} -\widetilde{\Sigma}_{\mathrm{sq}} \widetilde{V}^*V\|_F^2\nonumber\\
		&=&\|T\Sigma_{\mathrm{sq}} -\widetilde{\Sigma}_{\mathrm{sq}} T\|_F^2\\
		&=&\|T\Sigma_{\mathrm{sq}}\|_F^2+\|\widetilde{\Sigma}_{\mathrm{sq}} T\|_F^2-2\Re\, \tr\, T\Sigma_{\mathrm{sq}} T^*\widetilde{\Sigma}_{\mathrm{sq}}\nonumber\\
		&=&\sum_{j=1}^r\sigma_j^2+\sum_{j=1}^s\widetilde{\sigma}_j^2-2\sum_{i=1}^s\sum_{j=1}^r\widetilde{\sigma}_i\sigma_j\left|t_{ij}\right|^2\nonumber\\
		&=& F_{r,s}-2N,\nonumber\\
		\|E\|_F^2&=&F_{r,s}-2M.\nonumber
			\end{eqnarray*}
			Thus,
			\begin{eqnarray*}
				\dfrac{\|H-\widetilde{H}\|_F^2}{\|E\|_F^2}&=&\dfrac{F_{r,s}-2N}{F_{r,s}-2M}\\
				&\le&\dfrac{F_{r,s}-2N}{F_{r,s}-2G_{r,r}^\frac{1}{2}N^\frac{1}{2}}\\
				&=& \frac{1}{2G_{r,r}}\cdot\left(-\left(u+\dfrac{F_{r,s}^2-2G_{r,r}F_{r,s}}{u}\right)+2F_{r,s}   \right)\\
				 &&\text{ (where $u:=F_{r,s}-2G_{r,r}^\frac{1}{2}N^\frac{1}{2}\in [F_{r,s}-2G_{r,r}, F_{r,s}]$)}\\
				 &\le&\frac{1}{G_{r,r}}\cdot\left(F_{r,s}-\sqrt{F_{r,s}^2-2G_{r,r}F_{r,s}} \right).\\
				 &&\text{ (take $u=\sqrt{F_{r,s}^2-2G_{r,r}F_{r,s}}\in [F_{r,s}-2G_{r,r}, F_{r,s}]$)}\\ 
			\end{eqnarray*}

	To see sharpness.	Assume first that \(n-s\ge r\) and
		\(F_{r,s}>2G_{r,r}\). Set
		\[
		\Delta:=\sqrt{F_{r,s}^2-2G_{r,r}F_{r,s}},
		\qquad
		\alpha:=
		\frac{F_{r,s}}{F_{r,s}+\Delta}
		=
		\frac{F_{r,s}-\Delta}{2G_{r,r}}.
		\]
		Take \(S=I_m\). Choose the first \(r\) columns of \(T\) to be
		\[
		T_{(:,1:r)}=
		\begin{pmatrix}
		\alpha I_r\\
		0_{(s-r)\times r}\\
		\sqrt{1-\alpha^2}\,I_r\\
		0_{(n-s-r)\times r}
		\end{pmatrix},
		\]
		and complete these orthonormal columns to a unitary matrix \(T\).
		Then
		\[
		M=\alpha G_{r,r},
		\qquad
		N=\alpha^2G_{r,r},
		\]
		and direct substitution gives
		\[
		\frac{\|H-\widetilde H\|_F^2}{\|E\|_F^2}
		=
		\frac{F_{r,s}-2\alpha^2G_{r,r}}
		     {F_{r,s}-2\alpha G_{r,r}}
		=
		\frac{F_{r,s}-\Delta}{G_{r,r}}.
		\]

		If \(F_{r,s}=2G_{r,r}\) and \(n-s\ge r\), choose
		\(\alpha_\varepsilon:=1-\varepsilon\), \(0<\varepsilon<1\),
		and use the same unitary completion. Then
		\[
		\frac{\|H-\widetilde H\|_F^2}{\|E\|_F^2}
		=
		\frac{2G_{r,r}-2\alpha_\varepsilon^2G_{r,r}}
		     {2G_{r,r}-2\alpha_\varepsilon G_{r,r}}
		=
		1+\alpha_\varepsilon
		\longrightarrow 2.
		\]
		Notice also that when \(F_{r,s}=2G_{r,r}\), equality in the upper
		estimate would require
		\[
		N=G_{r,r},
		\qquad
		M=G_{r,r}.
		\]
		Consequently,
		\[
		\|E\|_F^2=F_{r,s}-2M=0.
		\]
		Thus no nonzero perturbation can attain equality in this case.
		Thus the sharp constant is obtained only as a limit in this case.
		
	For the lower bound, observe that
			\begin{eqnarray*}
				\dfrac{\|H-\widetilde{H}\|_F^2}{\|E\|_F^2}&=&\dfrac{F_{r,s}-2N}{F_{r,s}-2M}\\
				&\ge&\dfrac{F_{r,s}-2N}{F_{r,s}+2G_{r,r}^\frac{1}{2}N^\frac{1}{2}}\\
				&\ge&\dfrac{F_{r,s}-2G_{r,r}}{F_{r,s}+2G_{r,r}}.
			\end{eqnarray*}
			The equality can be attained by taking
			\begin{eqnarray*}
				S&=&\begin{pmatrix}
					-I_r&0\\
					0&\star
				\end{pmatrix},\quad T=\begin{pmatrix}
					I_r&0\\
					0&\star
				\end{pmatrix}.
			\end{eqnarray*}
			
			This completes the proofs.
			\qed

\medskip

\noindent\textbf{Proof of Theorems \ref{slee} and \ref{cslee}.} Consider the ratio
\begin{eqnarray*} 
		\dfrac{\|A+\widetilde{A}\|_F^2}{\|H+\widetilde{H}\|_F^2}&=&\dfrac{F_{r,s}+2M}{F_{r,s}+2N}\\
		&\le&\dfrac{F_{r,s}+2G_{r,r}^\frac{1}{2}N^\frac{1}{2}}{F_{r,s}+2N}\\
		&=& \frac{2G_{r,r}}{\left(u+\frac{F_{r,s}^2+2G_{r,r}F_{r,s}}{u} \right) -2F_{r,s}}\\
		&&\text{ (where $u:=F_{r,s}+2G_{r,r}^\frac{1}{2}N^\frac{1}{2}\in [F_{r,s}, F_{r,s}+2G_{r,r}]$)}\\
		&\le&\frac{G_{r,r}}{\sqrt{F_{r,s}^2+2G_{r,r}F_{r,s}} -F_{r,s}},\\
		&&\text{ (take $u=\sqrt{F_{r,s}^2+2G_{r,r}F_{r,s}}\in [F_{r,s}, F_{r,s}+2G_{r,r}]$)} .
	\end{eqnarray*}
To prove sharpness, assume \(n-s\ge r\), and set
\[
\Delta_+
:=
\sqrt{F_{r,s}^2+2G_{r,r}F_{r,s}},
\qquad
\alpha
:=
\frac{F_{r,s}}{F_{r,s}+\Delta_+}.
\]
Take \(S=I_m\). Choose the first \(r\) columns of \(T\) to be
\[
T_{(:,1:r)}
=
\begin{pmatrix}
	\alpha I_r\\
	0_{(s-r)\times r}\\
	\sqrt{1-\alpha^2}\,I_r\\
	0_{(n-s-r)\times r}
\end{pmatrix},
\]
and complete these orthonormal columns to a unitary matrix \(T\). Then
\[
M=\alpha G_{r,r},
\qquad
N=\alpha^2G_{r,r}.
\]
A direct substitution gives
\[
\frac{\|A+\widetilde A\|_F^2}
{\|H+\widetilde H\|_F^2}
=
\frac{F_{r,s}+2\alpha G_{r,r}}
{F_{r,s}+2\alpha^2G_{r,r}}
=
\frac{G_{r,r}}
{\Delta_+-F_{r,s}}.
\]
Thus the coefficient is attained whenever \(n-s\ge r\).
	
	Moreover, we have
	\begin{eqnarray*}
		\dfrac{\|A+\widetilde{A}\|_F^2}{\|H+\widetilde{H}\|_F^2}&=&\dfrac{F_{r,s}+2M}{F_{r,s}+2N}\\
		&\ge&\dfrac{F_{r,s}-2G_{r,r}^\frac{1}{2}N^\frac{1}{2}}{F_{r,s}+2N}\\
		&\ge&\dfrac{F_{r,s}-2G_{r,r}}{F_{r,s}+2G_{r,r}}.
	\end{eqnarray*}
It is easy to verify that the equality can be attained by taking
		\begin{eqnarray*}
		S&=&\begin{pmatrix}
			-I_r&0\\
			0&\star
		\end{pmatrix},\quad T=\begin{pmatrix}
			I_r&0\\
			0&\star
		\end{pmatrix}.
	\end{eqnarray*}\qed
	
\medskip

To facilitate the subsequent proofs,	let $A, B$ have the singular value decompositions
\[
A=U\Sigma V^*, B=\widehat{U}\widehat{\Sigma} \widehat{V}^*,
\]
where $U,\widehat{U}\in \mathbb{C}^{m\times m}, V, \widehat{V}\in \mathbb{C}^{n\times n}$ are unitary and 
\[
\Sigma=\begin{pmatrix}
	\Sigma_1& 0\\
	0&0
\end{pmatrix}\in \mathbb{C}^{m\times n}_r, \quad \widehat{\Sigma}=\begin{pmatrix}
	\widehat{\Sigma}_1& 0\\
	0&0
\end{pmatrix}\in \mathbb{C}^{m\times n}_s
\]
with $\Sigma_1=\diag(\sigma_1,\ldots,\sigma_r)$ and $\widehat{\Sigma}_1=\diag(\widehat{\sigma}_1,\ldots,\widehat{\sigma}_s)$.  Define
\[
\Sigma_{\mathrm{sq}}
:=
\operatorname{diag}
\bigl(
\sigma_1,\ldots,\sigma_r,
\underbrace{0,\ldots,0}_{n-r}
\bigr),
\qquad
\widehat{\Sigma}_{\mathrm{sq}}
:=
\operatorname{diag}
\bigl(
\widehat{\sigma}_1,\ldots,\widehat{\sigma}_s,
\underbrace{0,\ldots,0}_{n-s}
\bigr).
\]

We still denote $\widehat{S}=\widehat{U}^*U=(s_{ij}), \widehat{T}=\widehat{V}^*V=(t_{ij})$.

\medskip

\noindent\textbf{Proof of Theorem \ref{sag}.}
 Compute
	\begin{eqnarray*}
		\|AB^*\|_F^2&=&\|U\Sigma V^*\widehat{V}\widehat{\Sigma}^*\widehat{U}^*\|_F^2\\
		&=&\|\widehat{\Sigma}\widehat{T}\Sigma^*\|_F^2\\
		&=&\sum_{i=1}^s\sum_{j=1}^r\widehat{\sigma}_i^2\sigma_j^2\left|t_{ij}\right|^2 \\
		&\le& \sum_{j=1}^r\widehat{\sigma}_j^2\sigma_j^2,\\
		\|\left|A\right|^2+\left|B\right|^2\|_F^2&=&  \|V\Sigma_{\mathrm{sq}}^2V^*+\widehat{V}\widehat{\Sigma}_{\mathrm{sq}}^2\widehat{V}^*\|_F^2\\
		&=&\|\widehat{T}\Sigma_{\mathrm{sq}}^2+\widehat{\Sigma}_{\mathrm{sq}}^2\widehat{T}\|_F^2\\
		&=&\sum_{j=1}^r\sigma_j^4+\sum_{j=1}^s\widehat{\sigma}_j^4+2\sum_{i=1}^s\sum_{j=1}^r\widehat{\sigma}_i^2\sigma_j^2\left|t_{ij}\right|^2.\nonumber
	\end{eqnarray*}
	Here the first inequality follows because
	\((|t_{ij}|^2)_{1\le i\le s,\,1\le j\le r}\) is doubly
	substochastic, together with Lemma~\ref{lemext} and the
	rearrangement inequality. Moreover, the function
	\(x\mapsto x/(C+2x)\) is increasing for \(x\ge0\) whenever \(C>0\).
	Thus, \begin{eqnarray*}
		\dfrac{	\|AB^*\|_F^2}{	\|\left|A\right|^2+\left|B\right|^2\|_F^2}&=& \dfrac{\sum_{i=1}^s\sum_{j=1}^r\widehat{\sigma}_i^2\sigma_j^2\left|t_{ij}\right|^2}{\sum_{j=1}^r\sigma_j^4+\sum_{j=1}^s\widehat{\sigma}_j^4+2\sum_{i=1}^s\sum_{j=1}^r\widehat{\sigma}_i^2\sigma_j^2\left|t_{ij}\right|^2}\\
		&\le&  \frac{\sum_{j=1}^r\sigma_j^2\widehat{\sigma}_j^2}{\sum_{j=1}^r\sigma_j^4+\sum_{j=1}^s\widehat{\sigma}_j^4+2\sum_{j=1}^r\sigma_j^2\widehat{\sigma}_j^2}.
	\end{eqnarray*}
	Equivalently,
 \[
 \|AB^*\|_F\le \left(\frac{\sum_{j=1}^r\widehat{\sigma}_j^2\sigma_j^2}{\sum_{j=1}^r\sigma_j^4+\sum_{j=1}^s\widehat{\sigma}_j^4+2\sum_{j=1}^r\widehat{\sigma}_j^2\sigma_j^2} \right)^\frac{1}{2}\|\left|A\right|^2+\left|B\right|^2\|_F. 
 \]
If \(\widehat V=V\), then \(\widehat T=I_n\), and hence
\[
\sum_{i=1}^s\sum_{j=1}^r
\widehat{\sigma}_i^2\sigma_j^2|t_{ij}|^2
=
\sum_{j=1}^r\widehat{\sigma}_j^2\sigma_j^2.
\]
Thus equality holds in the displayed bound, and the coefficient is
optimal.
\qed

\medskip

\noindent\textbf{Proof of Theorem~\ref{scs}.}
By cyclicity of the trace,
\[
\operatorname{tr}(B^*A)
=
\sum_{i=1}^s\sum_{j=1}^r
\widehat{\sigma}_i\sigma_j
s_{ij}\overline{t_{ij}}.
\]
Hence, by the Cauchy--Schwarz inequality,
\begin{align*}
	\left|\operatorname{tr}(B^*A)\right|
	&\le
	\left(
	\sum_{i=1}^s\sum_{j=1}^r
	\widehat{\sigma}_i\sigma_j|s_{ij}|^2
	\right)^{1/2}\times
	\left(
	\sum_{i=1}^s\sum_{j=1}^r
	\widehat{\sigma}_i\sigma_j|t_{ij}|^2
	\right)^{1/2}.
\end{align*}
Both matrices
$
(|s_{ij}|^2)_{\substack{1\le i\le s\\1\le j\le r}}$
and $
(|t_{ij}|^2)_{\substack{1\le i\le s\\1\le j\le r}}
$
are doubly substochastic. Therefore, by Lemma~\ref{lemext} and the
rearrangement inequality,
\[
\sum_{i=1}^s\sum_{j=1}^r
\widehat{\sigma}_i\sigma_j|s_{ij}|^2
\le
\sum_{j=1}^r\sigma_j\widehat{\sigma}_j,
\]
and the same bound holds with \(s_{ij}\) replaced by \(t_{ij}\).
Consequently,
\[
\left|\operatorname{tr}(B^*A)\right|
\le
\sum_{j=1}^r\sigma_j\widehat{\sigma}_j.
\]
Since
$
\|A\|_F
=
\left(\sum_{j=1}^r\sigma_j^2\right)^{1/2}
$ and $
\|B\|_F
=
\left(\sum_{j=1}^s\widehat{\sigma}_j^2\right)^{1/2},
$
the stated formulation follows. Equality is attained when the left
and right singular vectors of \(A\) and \(B\) are aligned in the same
order.
\qed

\medskip

\noindent\textbf{Proof of Corollary~\ref{kitcor} and
	Theorems~\ref{thmlast} and~\ref{thmlast1}.}
If \(A=B\), all the asserted inequalities are immediate. Assume
henceforth that \(A\ne B\).

Let
\[
A=U\Lambda U^*,
\qquad
B=\widehat U\widehat\Lambda\widehat U^*,
\]
where
$
\Lambda
=
\operatorname{diag}
(\lambda_1,\ldots,\lambda_r,0,\ldots,0)
$
and
$
\widehat\Lambda
=
\operatorname{diag}
(\widehat\lambda_1,\ldots,\widehat\lambda_s,0,\ldots,0).
$
Put
\[
\widehat S:=\widehat U^*U=(s_{ij}),
\qquad
y_{ij}:=|s_{ij}|^2,
\quad
1\le i\le s,\quad 1\le j\le r.
\]
Then \(Y=(y_{ij})\in\mathcal C\), where $\mathcal C$ is defined by \eqref{eq:mathcalC}. A direct calculation gives
\[
\frac{\||A|-|B|\|_F^2}{\|A-B\|_F^2}
=
\frac{\Phi(Y)}{\Psi(Y)},
\]
where
$
	\Phi(Y)
	:=
	\widehat F_{r,s}
	-
	2\sum_{i=1}^s\sum_{j=1}^r
	|\widehat\lambda_i||\lambda_j|y_{ij},$ and 
$	\Psi(Y)
	:=
	\widehat F_{r,s}
	-
	2\sum_{i=1}^s\sum_{j=1}^r
	\Re\!\left(
	\widehat\lambda_i\overline{\lambda_j}
	\right)y_{ij}.
$

Since
$
\Re\!\left(
\widehat\lambda_i\overline{\lambda_j}
\right)
\le
|\widehat\lambda_i||\lambda_j|,
$
we have
$
0\le\Phi(Y)\le\Psi(Y).
$
Consequently,
\[
\||A|-|B|\|_F\le\|A-B\|_F,
\]
which proves Corollary~\ref{kitcor}.

By Lemma~\ref{lemext}, \(Y\) admits a representation
\[
Y=\sum_{\ell=1}^N\alpha_\ell P_\ell,
\qquad
P_\ell\in\operatorname{ext}(\mathcal C),
\qquad
\alpha_\ell\ge0,
\qquad
\sum_{\ell=1}^N\alpha_\ell=1.
\]
For every extreme point \(P_\ell\),
\(\Phi(P_\ell)\ge0\) and \(\Psi(P_\ell)\ge0\). Moreover,
\(\Psi(P_\ell)=0\) implies \(\Phi(P_\ell)=0\). Indeed, for a
subpermutation matrix these two quantities can be written as sums of
unmatched squared moduli and, respectively, matched terms
\[
|\widehat\lambda_i-\lambda_j|^2
\quad\text{and}\quad
\bigl(|\widehat\lambda_i|-|\lambda_j|\bigr)^2.
\]
It follows that
\[
\frac{\Phi(Y)}{\Psi(Y)}
=
\sum_{\substack{1\le\ell\le N\\\Psi(P_\ell)>0}}
\frac{\alpha_\ell\Psi(P_\ell)}{\Psi(Y)}
\frac{\Phi(P_\ell)}{\Psi(P_\ell)}.
\]
The coefficients on the right-hand side are nonnegative and sum to
\(1\). Hence the ratio is bounded below by its minimum over the zero
matrix and the rectangular subpermutation matrices. Evaluating these
extreme points gives
\[
\frac{\Phi(Y)}{\Psi(Y)}
\ge
\min_{\substack{0\le k\le r\\
		(i_1,\ldots,i_k)\in\mathfrak I_k(s)\\
		(j_1,\ldots,j_k)\in\mathfrak I_k(r)}}
\frac{
	\widehat F_{r,s}
	-2\displaystyle\sum_{t=1}^k
	|\widehat\lambda_{i_t}||\lambda_{j_t}|
}{
	\widehat F_{r,s}
	-2\displaystyle\sum_{t=1}^k
	\Re\!\left(
	\widehat\lambda_{i_t}\overline{\lambda_{j_t}}
	\right)
}.
\]
The zero matrix corresponds to \(k=0\) and gives the value \(1\).
This proves the asserted lower bound.

We next prove sharpness. Let \(P\) be a minimizing subpermutation
matrix. If its quotient is \(1\), then the lower bound just proved,
together with Corollary~\ref{kitcor}, forces equality, and the
coefficient is optimal.

Suppose now that the quotient \(q(P)\) is strictly less than \(1\) and
that \(P\) has fewer than \(r\) nonzero entries. Choose an unused row
\(i\) and an unused column \(j\), and set
$
P':=P+E_{ij}.
$
Write
\[
q(P)=\frac{n_0}{d_0},
\qquad
a:=|\widehat\lambda_i||\lambda_j|,
\qquad
b:=\Re\!\left(
\widehat\lambda_i\overline{\lambda_j}
\right).
\]
The denominator associated with \(P'\) is positive; otherwise equality
would hold in all the corresponding squared-difference terms, which
would imply \(q(P)=1\). Therefore,
\[
q(P')\le q(P)
\quad\Longleftrightarrow\quad
a\,d_0\ge b\,n_0.
\]
If \(b\le0\), the last inequality is immediate. If \(b>0\), then
\[
\frac{a}{b}\ge1>\frac{n_0}{d_0},
\]
and the same conclusion follows. Thus \(P\) can be extended without
increasing its quotient. Repeating this procedure produces a
minimizing full injection matrix. Every such matrix is the leading
\(s\times r\) block of a permutation unitary. Hence the lower
coefficient is attained and is optimal.

Finally, assume \(s=n\). In this case the columns of \(Y\) sum to \(1\).
The extreme points of the corresponding polytope are precisely the
full injection matrices
\[
\sum_{j=1}^r E_{\pi_jj},
\qquad
(\pi_1,\ldots,\pi_r)\in\mathfrak I_r(n).
\]
Applying the same weighted-ratio argument to the maximum gives the
coefficient in Theorem~\ref{thmlast1}. Each full injection is realized
by a permutation unitary, so the upper coefficient is also optimal.
\qed

\bigskip
\bibliographystyle{elsarticle-num}

\begin{thebibliography}{99}
	
	\bibitem{AY81}
	H. Araki and S. Yamagami,
	An inequality for Hilbert--Schmidt norm,
	\emph{Comm. Math. Phys.} \textbf{81} (1981), no.~1, 89--96,
	\doi{10.1007/BF01941801}.
	
	\bibitem{Bar89}
	A. Barrlund,
	Perturbation bounds on the polar decomposition,
	\emph{BIT} \textbf{30} (1990), no.~1, 101--113,
	\doi{10.1007/BF01932136}.
	
	\bibitem{Bha97}
	R. Bhatia,
	\emph{Matrix Analysis},
	Graduate Texts in Mathematics, vol.~169,
	Springer-Verlag, New York, 1997,
	\doi{10.1007/978-1-4612-0653-8}.
	
	\bibitem{BV04}
	S. Boyd and L. Vandenberghe,
	\emph{Convex Optimization},
	Cambridge University Press, Cambridge, 2004,
	\doi{10.1017/CBO9780511804441}.
	
	\bibitem{CL05}
	X. S. Chen and W. Li,
	Perturbation bounds on the polar decomposition under unitarily invariant norms,
	\emph{Math. Numer. Sinica} \textbf{27} (2005), no.~2, 121--128,
	\doi{10.12286/jssx.2005.2.121}.
	
	\bibitem{CL06}
	X. S. Chen and W. Li,
	Relative perturbation bounds for the subunitary polar factor under unitarily
	invariant norms,
	\emph{Adv. Math. (China)} \textbf{35} (2006), no.~2, 178--184.
	
	\bibitem{CL08}
	X. S. Chen and W. Li,
	Variations for the $Q$- and $H$-factors in the polar decomposition,
	\emph{Calcolo} \textbf{45} (2008), no.~2, 99--109,
	\doi{10.1007/s10092-008-0146-3}.
	
	\bibitem{CLS04}
	X. S. Chen, W. Li, and W. W. Sun,
	Some new perturbation bounds for the generalized polar decomposition,
	\emph{BIT} \textbf{44} (2004), no.~2, 237--244,
	\doi{10.1023/B:BITN.0000039423.58964.e1}.
	
	
	\bibitem{GV96}
	G. H. Golub and C. F. Van Loan,
	\emph{Matrix Computations},
	3rd ed., Johns Hopkins Studies in the Mathematical Sciences,
	Johns Hopkins University Press, Baltimore, MD, 1996.
	
	\bibitem{HLP34}
	G. H. Hardy, J. E. Littlewood, and G. P\'olya,
	\emph{Inequalities},
	Cambridge University Press, Cambridge, 1934.
	
	\bibitem{Hig86}
	N. J. Higham,
	Computing the polar decomposition---with applications,
	\emph{SIAM J. Sci. Statist. Comput.} \textbf{7} (1986), no.~4, 1160--1174,
	\doi{10.1137/0907079}.
	
	\bibitem{HMZ14}
	X. Hong, L. S. Meng, and B. Zheng,
	Some new perturbation bounds of generalized polar decomposition,
	\emph{Appl. Math. Comput.} \textbf{233} (2014), 430--438,
	\doi{10.1016/j.amc.2014.02.013}.
	
	\bibitem{KL91}
	C. Kenney and A. J. Laub,
	Polar decomposition and matrix sign function condition estimates,
	\emph{SIAM J. Sci. Statist. Comput.} \textbf{12} (1991), no.~3, 488--504,
	\doi{10.1137/0912027}.
	
	\bibitem{Kit86}
	F. Kittaneh,
	Inequalities for the Schatten $p$-norm. III,
	\emph{Comm. Math. Phys.} \textbf{104} (1986), no.~2, 307--310,
	\doi{10.1007/BF01211597}.
	
	\bibitem{Lee10}
	E.-Y. Lee,
	Rotfel'd type inequalities for norms,
	\emph{Linear Algebra Appl.} \textbf{433} (2010), no.~3, 580--584,
	\doi{10.1016/j.laa.2010.03.029}.
	
		\bibitem{LYS13}
	H. Li, H. Yang, and H. Shao,
	New perturbation bounds for nonnegative and positive polar factors,
	\emph{Math. Inequal. Appl.} \textbf{16} (2013), no.~2, 349--362,
	\doi{10.7153/mia-16-26}.

	
	\bibitem{LiR93}
	R.-C. Li,
	A perturbation bound for the generalized polar decomposition,
	\emph{BIT} \textbf{33} (1993), no.~2, 304--308,
	\doi{10.1007/BF01989752}.
	
	\bibitem{LiR95}
	R.-C. Li,
	New perturbation bounds for the unitary polar factor,
	\emph{SIAM J. Matrix Anal. Appl.} \textbf{16} (1995), no.~1, 327--332,
	\doi{10.1137/S0895479893256359}.
	
	\bibitem{LiR97}
	R.-C. Li,
	Relative perturbation bounds for the unitary polar factor,
	\emph{BIT} \textbf{37} (1997), no.~1, 67--75,
	\doi{10.1007/BF02510173}.
	
	\bibitem{LiR05}
	R.-C. Li,
	Relative perturbation bounds for positive polar factors of graded matrices,
	\emph{SIAM J. Matrix Anal. Appl.} \textbf{27} (2005), no.~2, 424--433,
	\doi{10.1137/S0895479803437153}.
	
	\bibitem{LiW05}
	W. Li,
	Some new perturbation bounds for subunitary polar factors,
	\emph{Acta Math. Sin. (Engl. Ser.)} \textbf{21} (2005), no.~6, 1515--1520,
	\doi{10.1007/s10114-004-0478-0}.
	
	\bibitem{LiW08}
	W. Li,
	On the perturbation bound in unitarily invariant norms for subunitary polar factors,
	\emph{Linear Algebra Appl.} \textbf{429} (2008), no.~2--3, 649--657,
	\doi{10.1016/j.laa.2008.03.025}.
	
	\bibitem{LS02}
	W. Li and W. W. Sun,
	Perturbation bounds of unitary and subunitary polar factors,
	\emph{SIAM J. Matrix Anal. Appl.} \textbf{23} (2002), no.~4, 1183--1193,
	\doi{10.1137/S0895479801394623}.
	
	\bibitem{LS03}
	W. Li and W. W. Sun,
	New perturbation bounds for unitary polar factors,
	\emph{SIAM J. Matrix Anal. Appl.} \textbf{25} (2003), no.~2, 362--372,
	\doi{10.1137/S0895479802413625}.
	
	\bibitem{LS06}
	W. Li and W. W. Sun,
	Some remarks on the perturbation of polar decompositions for rectangular matrices,
	\emph{Numer. Linear Algebra Appl.} \textbf{13} (2006), no.~4, 327--338,
	\doi{10.1002/nla.463}.
	
	\bibitem{LZ22}
	J. Lin and Y. Zhang,
	A proof of Lee's conjecture on the sum of absolute values of matrices,
	\emph{J. Math. Anal. Appl.} \textbf{516} (2022), no.~2,
	Paper No.~126542, 6 pp.,
	\doi{10.1016/j.jmaa.2022.126542}.
	
	\bibitem{Mat93}
	R. Mathias,
	Perturbation bounds for the polar decomposition,
	\emph{SIAM J. Matrix Anal. Appl.} \textbf{14} (1993), no.~2, 588--597,
	\doi{10.1137/0614041}.
	
	\bibitem{Rud76}
	W. Rudin,
	\emph{Principles of Mathematical Analysis},
	3rd ed., International Series in Pure and Applied Mathematics,
	McGraw-Hill Book Co., New York, 1976.
	
	\bibitem{SC89}
	J.-G. Sun and C.-H. Chen,
	Generalized polar decomposition,
	\emph{Math. Numer. Sinica} \textbf{11} (1989), no.~3, 262--273,
	\doi{10.12286/jssx.1989.3.262}.
	
	\bibitem{TZ26}
	Q. Tang and S. Zhang,
	Generalizing Lee's conjecture on the sum of absolute values of matrices,
	\emph{Linear Algebra Appl.} \textbf{731} (2026), 196--204,
	\doi{10.1016/j.laa.2025.11.015}.
	
	\bibitem{Zha11}
	F. Zhang,
	\emph{Matrix Theory: Basic Results and Techniques},
	2nd ed., Universitext,
	Springer, New York, 2011,
	\doi{10.1007/978-1-4614-1099-7}.
	
		
	\bibitem{ZYL13}
	P. Zhang, H. Yang, and H. Li,
	New perturbation bounds for the subunitary polar factors,
	\emph{Linear Multilinear Algebra} \textbf{61} (2013), no.~4, 517--526,
	\doi{10.1080/03081087.2012.693921}.
	
	\bibitem{Zha25}
	T. Zhang,
	A new proof of Lee's conjecture on the Frobenius norm via the
	matrix Cauchy--Schwarz inequality,
	\emph{Linear Algebra Appl.} \textbf{725} (2025), 355--358,
	\doi{10.1016/j.laa.2025.07.017}.
	
	\bibitem{Zha26a}
	T. Zhang,
	An operator triangle inequality for the quadratic symmetric modulus,
	arXiv:2602.01463 (2026),
	\doi{10.48550/arXiv.2602.01463}.
	
	\bibitem{Zha26b}
	T. Zhang,
	Operator symmetric moduli and sharp triangle inequalities,
	\emph{J. Lond. Math. Soc.} \textbf{114} (2026), no.~2, e70672,
	\doi{10.1112/jlms.70672}.
	
	\bibitem{Zhu18}
	L. Zhu, W.-W. Xu, H. Liu, and L.-J. Ma,
	New perturbation bounds in unitarily invariant norms for subunitary polar factors,
	\emph{Electron. J. Linear Algebra} \textbf{34} (2018), 231--239,
	\doi{10.13001/1081-3810.3528}.
	
\end{thebibliography}

\end{document}